\RequirePackage{fix-cm}
\documentclass[11pt, a4paper, oneside, pagebackref]{amsart}
\usepackage[latin1]{inputenc}
\usepackage[T1]{fontenc}
\usepackage{amssymb}
\usepackage[stretch=10]{microtype}
\usepackage{mathtools}
\usepackage[DIV=11, headinclude=true]{typearea}
\usepackage{xcolor}
\usepackage[hidelinks, linktoc=all]{hyperref}
 \usepackage{graphicx}

\providecommand{\href}[2]{#2}

\providecommand*{\backref}{}
\providecommand*{\backrefalt}{}
\renewcommand*{\backref}[1]{}
\renewcommand*{\backrefalt}[4]{%
	\ifcase #1 %
	\or
	  Cited page~#2.
	\else
	  Cited pages~#2.
	\fi
}

\newcommand\MTkillspecial[1]{
  \bgroup
  \catcode`\&=9
  \let\\\relax%
  \scantokens{#1}%
  \egroup
}
\newcommand\DeclarePairedDelimiterMultiline[3]{
  \DeclarePairedDelimiter{#1}{#2}{#3}
  \reDeclarePairedDelimiterInnerWrapper{#1}{star}{
    \mathopen{##1\vphantom{\MTkillspecial{##2}}\kern-\nulldelimiterspace\right.}
    ##2
    \mathclose{\left.\kern-\nulldelimiterspace\vphantom{\MTkillspecial{##2}}##3}}
}

\def\sumprime{\sideset{}{'}\sum}

\newcommand{\restr}{|_}

\newcommand{\E}{\mathbb{E}}
\newcommand{\Pbb}{\mathbb{P}}
\newcommand{\Z}{\mathbb{Z}}
\newcommand{\N}{\mathbb{N}}
\newcommand{\R}{\mathbb{R}}

\newcommand{\dd}{\mathop{}\!\mathrm{d}}
\DeclarePairedDelimiterMultiline{\pare}{(}{)}
\DeclarePairedDelimiterMultiline{\abs}{\lvert}{\rvert}
\DeclarePairedDelimiterMultiline{\norm}{\lVert}{\rVert}
\DeclarePairedDelimiterMultiline{\normLip}{\lVert}{\rVert_{\mathrm{Lip}}}
\DeclarePairedDelimiterMultiline{\acc}{\{}{\}}
\newcommand{\st}{\::\:}

\newcommand{\uHan}{\underline{D}_2^{\mathrm{an\vphantom{q}}}}
\newcommand{\oHan}{\overline{D}_2^{\mathrm{an\vphantom{q}}}}
\newcommand{\Han}{D_2^{\mathrm{an\vphantom{q}}}}
\newcommand{\uHqu}{\underline{D}_2^{\mathrm{qu}}}
\newcommand{\oHqu}{\overline{D}_2^{\mathrm{qu}}}
\newcommand{\Hqu}{D_2^{\mathrm{qu}}}

\DeclareMathOperator{\diam}{diam}
\DeclareMathOperator{\var}{var}

\DeclareMathOperator{\cLip}{Lip}
\newcommand{\coloneqq}{\mathrel{\mathop:}=}

\renewcommand{\epsilon}{\varepsilon}

\renewcommand{\phi}{\varphi}
\renewcommand{\leq}{\leqslant}
\renewcommand{\geq}{\geqslant}

\newtheorem{thm}{Theorem}[section]
\newtheorem{prop}[thm]{Proposition}
\newtheorem{definition}[thm]{Definition}
\newtheorem{lem}[thm]{Lemma}

\newtheorem*{prop*}{Proposition}

\theoremstyle{definition}

\newtheorem{rmk}[thm]{Remark}

\numberwithin{equation}{section}
\numberwithin{figure}{section}

\begin{document}

\title{Minimal distance between random orbits}

\author{S\'ebastien Gou\"ezel}

\address{S\'ebastien Gou\"ezel, IRMAR, CNRS UMR 6625,
Universit\'e de Rennes 1, 35042 Rennes, France}
\email{sebastien.gouezel@univ-rennes1.fr}

\author{J\'er\^ome Rousseau}
\address{J\'er\^ome Rousseau, CREC,Acad\'emie Militaire de St Cyr Co\"etquidan, 56381 GUER Cedex, France}
\address{Departamento de Matem\'atica, Universidade Federal da Bahia,
Av. Ademar de Barros s/n, 40170-110 Salvador, Brazil}
\address{IRMAR, CNRS UMR 6625,
Universit\'e de Rennes 1, 35042 Rennes, France}
\email{jerome.rousseau@ufba.br}

\author{Manuel Stadlbauer}
\address{Manuel Stadlbauer, Instituto de Matem\'atica,
Universidade Federal do Rio de Janeiro, Avenida Athos da Silveira Ramos  149, 21941-909 Rio de Janeiro (RJ), Brazil}
\email{manuel@im.ufrj.br}

\thanks{JR was partially supported by CNPq, by FCT projects PTDC/MAT-PUR/28177/2017 and
PTDC/MAT-PUR/4048/2021, with national funds, and by CMUP (UIDB/00144/2020),
which is funded by FCT with national (MCTES) and European structural funds
through the programs FEDER, under the partnership agreement PT2020. The third
author acknowledges financial support by Coordena\c{c}\~ao de Aperfei\c{c}oamento de
Pessoal de N\'{\i}vel Superior - Brasil (CAPES) - Finance Code 001 and by Conselho
Nacional de Desenvolvimento Cient\'{\i}fico - Brasil (CNPq) - PQ 312632/2018-5.}

\date{\today}



\begin{abstract}
We study the minimal distance between two orbit segments of length $n$, in
a random dynamical system with sufficiently good mixing properties. This
problem has already been solved in non-random dynamical system, and on
average in random dynamical systems (the so-called annealed version of the
problem): it is known that the asymptotic behavior for this question is
given by a dimension-like quantity associated to the invariant measure,
called its correlation dimension (or R\'enyi entropy). We study the analogous quenched question, and
show that the asymptotic behavior is more involved: two correlation dimensions
show up, giving rise to a non-smooth behavior of the associated asymptotic
exponent.
\end{abstract}

\maketitle

\section{Introduction}

\subsection{Main results}

This article is devoted to the study of the minimal distance between pieces
of orbits of length $n$, in a random dynamical system setting. By this, we
mean the following standard setting. We start from an invertible, probability
preserving dynamical system $(\Omega, \theta, \Pbb)$ on a compact metric
space, and consider another metric space $(X, d)$. For each $\omega \in
\Omega$, let $T_\omega$ be  a measurable map of $X$, such that the
skew-product map $S : (\omega, x)\mapsto (\theta \omega, T_\omega x)$ is
measurable and preserves a probability measure $\nu$ whose marginal on
$\Omega$ is $\Pbb$. The iterates of $S$ are given by $S^n (\omega, x) =
(\theta^n \omega, T_\omega^n x)$, where $T_\omega^n = T_{\theta^{n-1}
\omega}\circ \dotsm\circ T_\omega$ is a random composition of the
$T_{\omega}$'s, where the randomness is dictated by the driving map $\theta$.
In this setting, the measure $\nu$ can be disintegrated above $\Pbb$: there
is a family of probability measures $\mu_\omega$, depending measurably on
$\omega$, such that for any bounded function $f$ holds $\int f\dd\nu = \int
\pare*{\int f(x) \dd\mu_\omega(x)} \dd\Pbb(\omega)$. We write informally $\nu
= \Pbb \otimes \mu_\omega$. Let $\mu = \int \mu_\omega \dd\Pbb(\omega)$ be
the second marginal of $\nu$. As $\nu$ is invariant under $S$, the measures
$\mu_\omega$ also satisfy an invariance property: $(T_\omega)_* \mu_\omega =
\mu_{\theta \omega}$ for $\Pbb$-a.e.\ $\omega$.

We are interested in the minimal distance between two pieces of orbit of
length $n$. In a classical dynamical system setting, this would amount to
understanding the behavior of $\min_{i, j < n} d (T^i x, T^j y)$ for a
typical pair $(x, y)$. It has been shown in~\cite{BaLiRo} that the rate of decay
to zero of this quantity is related to a dimension-like quantity associated
to the invariant measure, called its correlation dimension (or R\'enyi entropy for symbolic dynamical systems), measuring the polynomial
decay rate of the $r$-neighborhood of the diagonal in $X \times X$ in terms
of $r$.

In the random dynamics situation, there are two possible interpretations for
this question. One may consider typical pairs $(\omega, x)$ and $(\omega',
y)$ and try to minimize $d(T_\omega^i x, T_{\omega'}^j y)$. This is the
annealed question, where randomness is taken over the whole product space
$\Omega \times X$. It has already been studied in~\cite{colaro}, and the outcome
is comparable to the situation of classical dynamical systems (the relevant
quantity being the correlation dimension of the second marginal $\mu$ of $\nu$). One
may also consider a typical $\omega$, and then for this fixed $\omega$ pick a
typical pair $(x, y)$ for $\mu_\omega$ and try to minimize $d(T_\omega^i x,
T_\omega^j y)$. This is the quenched variant of our main question, to which
this article is devoted.

In many questions about random systems, the outcome in a quenched situation
is similar to the outcome of the annealed situation, but harder to prove.
This is not the case here: we observe a behavior which is genuinely different
from the annealed case, with a phase transition: there are two competing
phenomena to decide the decay rate of the minimal distance between orbits,
one similar to the annealed situation and one that is specific to the
quenched situation, and each of them can be prevalent in some situations.

We denote
\begin{equation*}
  \uHan = \liminf_{r\to 0} \frac{\log \int \mu(B(x,r)) \dd\mu(x)}{\log r}, \quad
  \oHan = \limsup_{r\to 0} \frac{\log \int \mu(B(x,r)) \dd\mu(x)}{\log r}
\end{equation*}
for the lower and upper correlation dimensions of the measure $\mu$. The $2$
in the notation comes from the fact that this in an $L^2$-like expression,
which is easier to see in a symbolic setting as in Remark~\ref{rmk:lcs}.
These are annealed quantities, referring to the averaged measure $\mu = \int
\mu_\omega \dd\Pbb(\omega)$, hence the superscript $\mathrm{an}$. In the
quenched version, one should rather compute the correlation dimension of each
measure $\mu_\omega$ and then average with respect to $\Pbb$, giving rise to
the following definitions:
\begin{align*}
  \uHqu &= \liminf_{r\to 0} \frac{\log \int \mu_\omega(B(x,r)) \dd\mu_\omega(x) \dd\Pbb(\omega)}{\log r},
  \\ \oHqu &= \limsup_{r\to 0} \frac{\log \int \mu_\omega(B(x,r)) \dd\mu_\omega(x) \dd\Pbb(\omega)}{\log r}.
\end{align*}
When the liminf and the limsup coincide, we denote the corresponding
quantities by $\Han$ and $\Hqu$.

Our main theorem shows that the decay rate of the minimal distance between
orbits, in the quenched situation, can be expressed in terms of $\Han$ and
$\Hqu$. This result requires that the geometry of the space should be nice
enough (spaces with bounded local complexity, see
Definition~\ref{def:local_complexity} below -- this is a very mild geometric
condition on the space, satisfied for instance by shift spaces and Riemannian
manifolds), that the measures $\mu_\omega$ depend in a Lipschitz way on
$\omega$ (see Definition~\ref{def:lipschitz_fibers}) and that the system
mixes quickly enough, both for the base map and the fiber maps (stretched
exponential mixing, see Definitions~\ref{def:4mixing_base}
and~\ref{def:mixing_fiber}). Finally, we also require that the map $S$ is
Lipschitz.

\begin{thm}
\label{thm:main} Let $X$ be a compact metric space with bounded local
complexity. Consider a random dynamical system $S : \Omega \times X \to
\Omega \times X$ preserving a probability measure $\nu$, for which $\Han$ and
$\Hqu$ are well defined, and for which $\omega \to \mu_\omega$ is Lipschitz.
Assume that $S$ is Lipschitz, has fiberwise stretched exponential mixing, and
that the base map has stretched exponential $4$-mixing. Then, for
$\Pbb$-almost every $\omega$, for $\mu_\omega^{\otimes 2}$-almost every $x,
y$, one has the convergence
\begin{equation}
\label{eq:main}
  \frac{-\log \min_{i, j < n} d(T_\omega^i x, T_\omega^j y)}{\log n} \to \max\acc*{\frac{2}{\Han}, \frac{1}{\Hqu}}.
\end{equation}
\end{thm}

This theorem should be compared with the corresponding statement in the
annealed situation: under the same assumptions, for $\nu^{\otimes 2}$
almost-every pairs $(\omega, x), (\omega', y)$, one has
\begin{equation*}
  \frac{-\log \min_{i, j < n} d(T_\omega^i x, T_{\omega'}^j y)}{\log n} \to \frac{2}{\Han},
\end{equation*}
by~\cite[Theorem 4.4]{colaro}.

Theorem~\ref{thm:main} is a consequence of several statements on upper and
lower bounds, which for some of them require weaker assumptions regarding
mixing, and which can be expressed in terms of $\uHan$, $\uHqu$ and $\oHan$,
$\oHqu$ respectively, without requiring that $\Han$ and $\Hqu$ are well
defined. These more precise versions are discussed in
Paragraph~\ref{subsec:fine-grained}, after the precise meaning of our
assumptions is discussed in the next paragraph.

Let us stress that Theorem~\ref{thm:main} applies to a large class of
concrete uniformly and nonuniformly expanding random dynamical systems, see
Paragraphs~\ref{subsec:Finitely many Ruelle expanding maps}
and~\ref{subsec:Non-uniformly expanding local diffeomorphisms}. Also, one can
construct examples in which the maximum in the right hand side
of~\eqref{eq:main} is realized either by the first or the second term. See in
particular Paragraph~\ref{subsec:transition} in which we exhibit a family of
systems depending smoothly on a parameter for which there is a transition
from the first behavior to the second behavior, in a non-smooth way,
exhibiting a second-order phase transition for the minimal approximation rate
of orbits in the quenched setting (while there is no such phase transition
for the analogous annealed question).

When $\Omega$ is a point, the annealed and quenched correlation dimensions coincide,
so the maximum is always realized by $2/\Han$. This is also the case when we
are close enough to a product situation (in which case all the $\mu_\omega$
are close to $\mu$), but $1/\Hqu$ may become dominant in more distorted
situations, as testified in Paragraph~\ref{subsec:transition}.

The intuition as to which term is dominant is the following. If one considers
$i$ far away from $j$, then $\theta^i \omega$ and $\theta^j \omega$ are
essentially independent, so $T_\omega^i x$ and $T_\omega^j y$ are essentially
two independent points distributed according to $\mu$, and one should get the
same behavior as in the annealed situation. There are $n^2$ such pairs $(i,
j)$, and for each of them the probability that the points are close by is
governed by the dimension $\Han$, hence an asymptotics $2/\Han$. For $j = i$ on
the other hand, the points $T_\omega^i x$ and $T_\omega^i y$ are independent
points distributed according to the measure $\mu_{\theta^i \omega}$, so the
correlation dimension of this measure should appear in the asymptotic. Since there
are only $n$ such pairs $(i, i)$ (as opposed to $n^2$ before), we get an
asymptotics $1/\Hqu$. The precise statements in
Paragraph~\ref{subsec:fine-grained} will make this intuition precise, by
showing that the on-diagonal and off-diagonal behaviors are genuinely
different.

It is interesting to specialize Theorem~\ref{thm:main} to the case of a
deterministic dynamical system (taking $\Omega$ to be a point). Many of our
assumptions become trivial in this situation. The statement becomes the
following.
\begin{thm}
\label{thm:deterministic} Let $T:X\to X$ be a Lipschitz map on a compact
metric space with bounded local complexity, preserving a probability measure
$\mu$ with a well-defined correlation dimension $D_2(\mu)$. Assume that $T$
mixes stretched exponentially. Then, for $\mu^{\otimes 2}$-almost all $x, y$,
\begin{equation*}
  \frac{-\log \min_{i, j < n} d(T^i x, T^j y)}{\log n} \to \frac{2}{D_2(\mu)}.
\end{equation*}
\end{thm}

This theorem is essentially proved in~\cite{BaLiRo}, although the assumptions
there are phrased in a slightly different way.

\begin{rmk}[Longest common substring]\label{rmk:lcs}
When the (random) dynamical system is a (random) shift, i.e.,
$X=\mathcal{A}^\N$ for some alphabet $\mathcal{A}$ and $T=\sigma$ (or
$T_\omega=\sigma$) with $\sigma$ the left shift, it was observed
in~\cite{BaLiRo} that studying the minimal distance between orbits is
equivalent to studying the length of the longest common substring between two
sequences, that is:
\begin{equation*}
-\log \min_{i, j < n} d(\sigma^i x, \sigma^j y)
= \max\{m: \exists 0\leq i,j< n \textrm{ s.t. } x_{i+k}=y_{j+k}\textrm{ for }k=0,\dotsc,m-1 \}.
\end{equation*}
In this case, balls will correspond to cylinders and the correlation
dimensions coincide with the annealed and quenched R\'enyi entropies
\[
{H}^\textrm{an}_2 = \underset{k \to \infty}{{\lim}}\frac{\log\sum \mu(C_k)^2
}{-k} \textrm{ and } { H}^\textrm{qu}_2 = \underset{k \to
\infty}{{\lim}}\frac{\log\sum \int \mu_\omega(C_k)^2d\Pbb(\omega)}{-k},
\]
where the sums are taken over all k-cylinders.
\end{rmk}
\subsection{The technical assumptions}

In this paragraph, we specify precisely the technical assumptions made in
Theorem~\ref{thm:main}. The various assumptions will also be useful to
highlight, in Paragraph~\ref{subsec:fine-grained}, which statements require
stronger or weaker assumptions.

A function $f : X \to \R$ is Lipschitz if it satisfies the inequality
$\abs{f(x)-f(y)}\leq C d(x,y)$ for all $x,y$. The best such $C$ is called the
Lipschitz constant of $f$ and denoted by $\cLip(f)$. We define the Lipschitz
norm of $f$, denoted by $\normLip{f}$, to be the sum of its sup norm and its
Lipschitz constant. In this way, $\normLip{fg} \le \normLip{f} \normLip{g}$.

Here is our main geometric assumption on the spaces we consider.
\begin{definition}
\label{def:local_complexity} A compact metric space $X$ has bounded local
complexity if there exists a constant $C_0$ such that, for any small enough
$r$, there exist a constant $k(r)<+\infty$ and points $x^{(r)}_1, \dotsc, x^{(r)}_{k(r)}$ in $X$ such that
the space is covered by the balls $(B(x^{(r)}_p, r))_{1\le p\le k(r)}$ and
any point $x$ belongs to at most $C_0$ balls $B(x^{(r)}_p, 4 r)$.
\end{definition}
Basic examples are shift spaces on finitely many symbols: for these, one may
take the balls $B(x^{(r)}_p, r)$ as the different cylinders of a given length
$N$, and they are all disjoint. Compact Riemannian manifolds have also
bounded local complexity: this follows from the fact that Euclidean spaces
are, using finitely many charts and an approximation argument to reduce to
this situation.

\begin{definition}
\label{def:lipschitz_fibers} Given a random dynamical system on $\Omega
\times X$, the random fiber measures $\mu_\omega$ depend on a Lipschitz way
on $\omega$ if there exists $C_1>0$ such that, for any Lipschitz function $f
: X \to \R$, for any $\omega, \omega'$,
\begin{equation*}
  \abs*{\int f \dd\mu_\omega - \int f \dd\mu_{\omega'}} \leq C_1 \normLip{f} d(\omega, \omega').
\end{equation*}
\end{definition}

Let us now turn to the various mixing conditions we need, for the base map or
the fiber maps.

\begin{definition}
\label{def:mixing_base} The dynamical system $\theta : \Omega \to \Omega$
mixes stretched exponentially if there exist $c_2 > 0$ and $C_2 > 0$ such
that, for any Lipschitz functions $f, g : \Omega \to \R$, for any $n\in \N$,
\begin{equation*}
  \abs*{\int f \cdot g\circ \theta^n \dd\Pbb - \pare*{\int f \dd\Pbb} \pare*{\int g\dd\Pbb}}
  \leq C_2 e^{-n^{c_2}} \normLip{f} \normLip{g}.
\end{equation*}
\end{definition}

We will need a stronger property, ensuring that there is quantitative mixing
for $4$ functions instead of $2$, if there is a large enough time gap between
the second and third functions. This property, that we call stretched
exponential $4$-mixing, implies the usual stretched exponential mixing of
Definition~\ref{def:mixing_base} (take $f_2 = 1$ and $g_1 = 1$)

\begin{definition}
\label{def:4mixing_base} The dynamical system $\theta : \Omega \to \Omega$
has stretched exponential $4$-mixing if there exist $c_2 > 0$ and $C_2 > 0$
such that, for any Lipschitz functions $f_1, f_2, g_1, g_2 : \Omega \to \R$,
for any $n\in \N$, for any $a\le b \le c$ with $b - a \ge n$,
\begin{multline}
\label{eq:4mixing}
  \abs*{\int f_1 \cdot f_2\circ \theta^a \cdot g_1\circ \theta^b \cdot g_2 \circ \theta^c  \dd\Pbb - \pare*{\int f_1 \cdot f_2 \circ \theta^a \dd\Pbb} \pare*{\int g_1 \cdot g_2 \circ \theta^{c-b}\dd\Pbb}}
  \\
  \leq C_2 e^{-n^{c_2}} \normLip{f_1} \normLip{f_2} \normLip{g_1} \normLip{g_2}.
\end{multline}
\end{definition}

Finally, we give a fiberwise mixing condition. In the case where $\Omega$ is
a point (i.e., for a deterministic dynamical system), as in
Theorem~\ref{thm:deterministic}, this is the only nontrivial assumption.

\begin{definition}
\label{def:mixing_fiber} The random dynamical system $S : \Omega \times X \to
\Omega \times X$ mixes stretched exponentially along the fibers if there
exist $c_3 > 0$ and $C_3 > 0$ such that, for any Lipschitz functions $f, g :
X \to \R$, for any $n\in \N$, for any $\omega \in \Omega$
\begin{equation*}
  \abs*{\int f \cdot g\circ T_\omega^n \dd\mu_\omega - \pare*{\int f \dd\mu_\omega} \pare*{\int g\dd\mu_{\theta^n \omega}}}
  \leq C_3 e^{-n^{c_3}} \normLip{f} \normLip{g}.
\end{equation*}
\end{definition}

\subsection{More fine-grained results}
\label{subsec:fine-grained}

Let $m_n(\omega; x, y) = \min_{i,j<n} d(T_\omega^i x, T_\omega^j y)$ be the
minimal distance between orbit segments of length $n$. For more precise
results, we will need to split it further according to the allowed gap
between $i$ and $j$. Accordingly, let
\begin{equation}
\label{eq:def_alpha}
  \alpha(n) = (\log n)^{C_4},
\end{equation}
where $C_4$ is large enough (we will need $C_4
\ge \max(2/c_2, 2/c_3)$, where $c_2$ and $c_3$ are the rates of stretched
exponential mixing along the basis and the fibers respectively). Let
\begin{align*}
  m_n^0(\omega; x, y) &= \min_{i < n} d(T_\omega^i x, T_\omega^i y),\\
  m_n^{\le}(\omega; x, y) &= \min_{\substack{i,j < n\\ \abs{j-i}\le \alpha(n)}} d(T_\omega^i x, T_\omega^j y),\\
  m_n^{>}(\omega; x, y) &= \min_{\substack{i,j < n\\ \abs{j-i}> \alpha(n)}} d(T_\omega^i x, T_\omega^j y),\\
  m_n^{\gg}(\omega; x, y) &= \min_{i< n/3,\ 2n/3\le j<n} d(T_\omega^i x, T_\omega^j y).
\end{align*}

We start with the upper bounds for $-\log m_n$, i.e., with the lower bounds
for $m_n$: we have to show that the orbits are never too close to each other.
For this, we will split $m_n$ as $\min(m_n^{\le}, m_n^>)$ and show separately
that these two terms are almost surely not too small.

\begin{prop}
\label{prop:Mnle_upper_bound} Assume that the space $X$ has bounded local
complexity. Then, for $\Pbb$-almost every $\omega$, for $\mu_\omega^{\otimes
2}$ every $x,y$, one has
\begin{equation*}
  \limsup_{n\to \infty} \frac{-\log m_n^{\le}(\omega; x,y)}{\log n} \le \frac{1}{\uHqu}.
\end{equation*}
\end{prop}

\begin{prop}
\label{prop:Mngt_upper_bound} Assume that the space $X$ has bounded local
complexity, that the fiber measures $\mu_\omega$ depend in a Lipschitz way on
$\omega$, and that $\theta$ mixes stretched exponentially. Then, for
$\Pbb$-almost every $\omega$, for $\mu_\omega^{\otimes 2}$ every $x,y$, one
has
\begin{equation*}
  \limsup_{n\to \infty} \frac{-\log m_n^{>}(\omega; x,y)}{\log n} \le \frac{2}{\uHan}.
\end{equation*}
\end{prop}

Combining the two previous propositions, and since $m_n = \min (m_n^\le,
m_n^>)$, one obtains almost surely
\begin{equation}
\label{eq:upper_bound}
  \limsup_{n\to \infty} \frac{-\log m_n(\omega; x,y)}{\log n} \le
  \max\acc*{\frac{2}{\uHan}, \frac{1}{\uHqu}},
\end{equation}
proving the first (easy) half of Theorem~\ref{thm:main}.

\medskip

Let us now deal with the lower bounds for $-\log m_n$, i.e., with the upper
bounds for $m_n$: we have to show that there are some times at which the
orbits are pretty close. We can select those times as we like. We will use
either $i=j$ (given by $m_n^0$) or $i$ and $j$ very far apart, i.e., $i<n/3$
and $2n/3\le j<n$ (given by $m_n^{\gg}$). In other words, we use the trivial
inequality $m_n \le \min (m_n^0, m_n^\gg)$, and we will get good upper bounds
for these two terms.

\begin{prop}
\label{prop:Mn0_lower_bound} Assume that the space $X$ has bounded local
complexity, that the fiber measures $\mu_\omega$ depend in a Lipschitz way on
$\omega$, that $\theta$ mixes stretched exponentially and that $S$ mixes
stretched exponentially along the fibers. Then, for $\Pbb$-almost every
$\omega$, for $\mu_\omega^{\otimes 2}$ every $x,y$, one has
\begin{equation*}
  \liminf_{n\to \infty} \frac{-\log m_n^0(\omega; x,y)}{\log n} \ge \frac{1}{\oHqu}.
\end{equation*}
\end{prop}

\begin{prop}
\label{prop:Mngg_lower_bound} Assume that the space $X$ has bounded local
complexity, that the fiber measures $\mu_\omega$ depend in a Lipschitz way on
$\omega$, that $\theta$ has stretched exponential $4$-mixing and that $S$
mixes stretched exponentially along the fibers. Assume also that $S$ is
Lipschitz. Then, for $\Pbb$-almost every $\omega$, for $\mu_\omega^{\otimes
2}$ every $x,y$, one has
\begin{equation*}
  \liminf_{n\to \infty} \frac{-\log m_n^\gg(\omega; x,y)}{\log n} \ge \frac{2}{\oHan}.
\end{equation*}
\end{prop}

Combining the two previous propositions, and since $m_n \leq \min (m_n^0,
m_n^\gg)$, one obtains almost surely
\begin{equation*}
  \liminf
  _{n\to \infty} \frac{-\log m_n(\omega; x,y)}{\log n} \ge
  \max\acc*{\frac{2}{\oHan}, \frac{1}{\oHqu}},
\end{equation*}
proving the second (harder) half of Theorem~\ref{thm:main}.

The proofs of all these theorems are given in Section~\ref{sec:proofs}, after
several examples are discussed in Section~\ref{sec:examples}. They go from
the easiest one (Proposition~\ref{prop:Mnle_upper_bound}) to the hardest one
(Proposition~\ref{prop:Mngg_lower_bound}). The results on the upper bounds
(Propositions~\ref{prop:Mnle_upper_bound} and~\ref{prop:Mngt_upper_bound})
are given in Paragraph~\ref{ssubec:proof_upper}. They are based on a first
moment computation. The results on the lower bounds
(Propositions~\ref{prop:Mn0_lower_bound} and~\ref{prop:Mngg_lower_bound}) are
then established in Paragraph~\ref{subsec:proof_lower}. They rely on
technically more involved second moment estimates, that require stronger
mixing conditions.

\section{Examples}
\label{sec:examples} We now present an explicit example with a phase
transition, as well as two classes of random dynamical systems which satisfy
the hypothesis of Theorem~\ref{thm:main}. The first class is constructed from
finitely many uniformly expanding maps whereas the second is given by a
continuous family of non-uniformly expanding, local diffeomorphisms on a
manifold.

\subsection{An explicit example with a phase transition}
\label{subsec:transition} We construct in this section a simple random
Bernoulli shift where, depending on the sample measures, we will obtain a
phase transition. As explained in Remark~\ref{rmk:lcs}, in this example we
will compute R\'enyi entropies which correspond to correlation dimensions on
shift spaces.

Let $(\Omega,\theta)$ be the full shift on the symbolic space
$\Omega=\{A,B\}^\Z$ and let $\Pbb=\mathcal{P}^\Z$ with
$\mathcal{P}(A)=\mathcal{P}(B)=\frac{1}{2}$. We then consider the one-sided
shift on $X=\{0,1\}^\N$ as a random subshift by constructing a random
Bernoulli measure as follows. Let $p_A,p_B\in(0,1)$. The random Bernoulli
measure $\{ \mu_\omega : \omega \in \Omega \} $ is defined by
\[
\mu_\omega([x_0,\dotsc, x_n])=\mu_{\omega_0}(x_0)\mu_{\omega_1}(x_1)\dotsm
\mu_{\omega_n}(x_n),
\]
with $\mu_A(0)=p_A$ and $\mu_A(1)=1-p_A$ on the one hand, and $\mu_B(0)=p_B$
and $\mu_B(1)=1-p_B$ on the other hand.

To compute the R\'enyi entropy, observe that for a cylinder $C_n=[x_0,\dotsc,
x_{n-1}]$
\begin{align*}
\mu(C_n) &=
\int\mu_\omega([x_0,\dotsc, x_{n-1}])d\Pbb(\omega)
= \int \prod_{i=0}^{n-1} \mu_{\omega_i}(x_i) d\Pbb(\omega)
= \prod_{i=0}^{n-1} \int \mu_{\omega_i}(x_i) d\Pbb(\omega_i)
\\
&= \prod_{i=0}^{n-1} \left(\frac{1}{2}\mu_A(x_i)+\frac{1}{2}\mu_B(x_i)\right)
= \frac{1}{2^n} \left(p_A + p_B\right)^{\# \{i: x_i =0\}} \left(2 - p_A - p_B\right)^{\# \{i: x_i =1\}}.
\end{align*}
Thus, by the binomial identity,
\[\sum_{C_n}\mu(C_n)^2= \frac{1}{2^{2n}}
\left(\left(p_A + p_B\right)^2+\left( 2-p_A -p_B\right)^2\right)^{n}
= \frac{1}{2^{n}}
\left(  \left(p_A + p_B\right)^2 -2(p_A + p_B) + 2 \right)^{n},
\]
which implies that
\[
H_2^{\textrm{an}} = -\log\left(\frac{1}{2}\left(  \left(p_A + p_B\right)^2 -2(p_A + p_B) + 2\right)\right).
\]
Moreover, it follows by the same arguments that
\begin{align*}
\int\mu_\omega(C_n)^2 d\Pbb(\omega)
&= \prod_{i=0}^{n-1} \int \mu_{\omega_i}(x_i)^2  d\Pbb(\omega_i)
=   \prod_{i=0}^{n-1} \left(\frac{1}{2} \mu_A(x_i)^2+\frac{1}{2}\mu_B(x_i)^2\right)
\\ & =  \frac{1}{2^n}\left(p_A^2 + p_B^2\right)^{\# \{i: x_i =0 \}} \left((1 - p_A)^2 + (1-p_B)^2 \right)^{\# \{i: x_i =1\}}.
\end{align*}
Thus,
\[\sum_{C_n}\int\mu_\omega(C_n)^2d\Pbb(\omega)=\left(\frac{1}{2}\left(p_A^2+p_B^2+(1-p_A)^2+(1-p_B)^2\right)\right)^{n}\]
and
\begin{align*}
H_2^\textrm{qu}
& =-\log\left(\frac{1}{2}\left(p_A^2+p_B^2+(1-p_A)^2+(1-p_B)^2\right)\right)
\\& = -\log\left(\left(p_A^2 - \tfrac12\right)^{2} +\left(p_B^2 -
\tfrac12\right)^{2} + \tfrac12  \right) .
\end{align*}
Proposition~\ref{prop:ruelle-expanding-semigroup} below shows that the
hypothesis of Theorem~\ref{thm:main} are satisfied. Hence, for $\Pbb$-almost
every $\omega$, for $\mu_\omega \otimes \mu_\omega$-almost every $(x,y)$,
\begin{multline*}
\frac{-\log m_n(\omega;x,y)}{\log n} \xrightarrow[n\to\infty]{}
\max\left\{\frac{2}{H_2^{\textrm{an}}},\frac{1}{H_2^\textrm{qu}}\right\} \\
= \max \left\{
\frac{2}{-\log \frac{1}{2} \left( \left(p_A + p_B\right)^2 -2(p_A + p_B) + 2 \right)}, \frac{1}{-\log\left(\left(p_A^2 - \tfrac12\right)^{2} +\left(p_B^2 - \tfrac12\right)^{2} + \tfrac12  \right)}
\right\}.
\end{multline*}
The behavior of this maximum depends on the values of $p_A$ and $p_B$ and
some simple choices can give us distinctive behaviors.  If $p_A$ and $p_B$
are close enough (for example if $\abs{p_A-p_B}\leq\frac{1}{2} $) then
$\max\left\{{2}/{H_2^\textrm{an}},{1}/{H_2^\textrm{qu}}\right\}={2}/{H_2^\textrm{an}}$.
In this case, we obtain the same behaviour as observed
in~\cite[Example~2.1]{lcs-random}. However, if $p_A$ and $p_B$ are
sufficiently far from each other, there is a phase transition as the quenched
parameter becomes dominant. In order to obtain a precise description of these
domains, it suffices to determine the separating curve given by
$H_2^\textrm{an}(p_A,p_B) = 2 H_2^\textrm{qu}(p_A,p_B)$ (see the left hand
side of Figure~\ref{figure:phase-transition}). The same argument then gives
rise to the contours (or level sets) of the function $(p_A , p_B) \mapsto
\max\left\{{2}/{H_2^\textrm{an}},{1}/{H_2^\textrm{qu}}\right\}$, which are
displayed on the right hand side of Figure~\ref{figure:phase-transition}.
Observe that the formulas for $H_2^\textrm{an}$ and $H_2^\textrm{qu}$ imply
that these contours consist of straight lines of slope $-1$ and circle
segments with center $(1/2,1/2)$.
\begin{center}
\begin{figure}[h]\label{figure:phase-transition}
\includegraphics[width=0.9\textwidth]{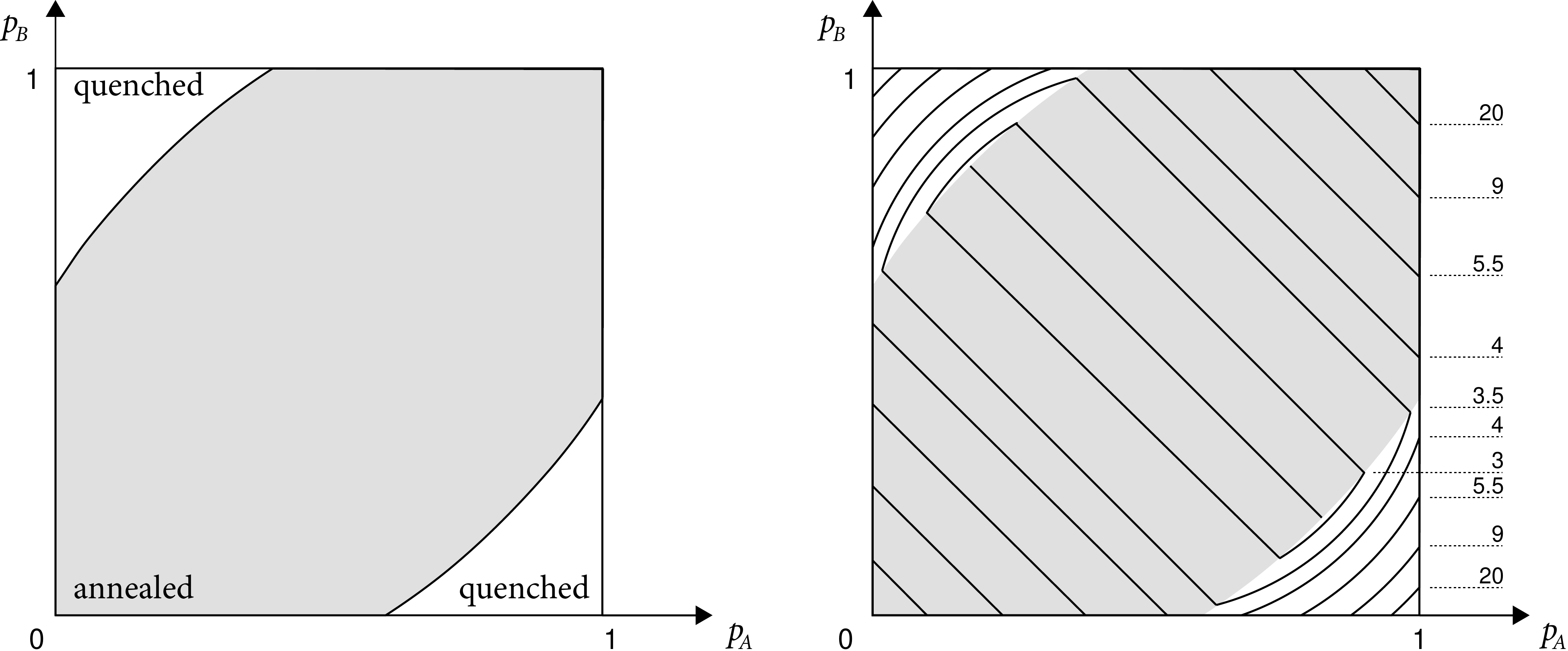}
\caption{Regions with quenched and annealed behaviour}
\end{figure}
\end{center}
For example, if we choose $p_B=1-p_A$, and set $c_\pm \coloneqq
\frac{1}{2} \pm \frac{1}{2} \sqrt{ 2 \sqrt{\sqrt{2} -1}  - 1 }$, then
\[\underset{n \to \infty}{{\lim}}\frac{-\log m_n(\omega;x,y)}{\log n} =\max\left\{\frac{2}{H_2^\textrm{an}},\frac{1}{H_2^\textrm{qu}}\right\}=
\begin{cases}
 \frac{1}{H_2^\textrm{qu}} &: \phantom{_+}0<p_A \leq c_- \\
 \frac{2}{H_2^\textrm{an}} &: c_-< p_A \leq c_+ \\
 \frac{1}{H_2^\textrm{qu}} &: c_+ < p_A \leq 1.
\end{cases} \]
To illustrate this phase transition, the graph of $p_A \mapsto  \max\left\{ {2}/{H_2^\textrm{an}},{1}/{H_2^\textrm{qu}}\right\}$, provided that $p_A=1-p_B$, is presented in Figure \ref{fig:graph}.
\begin{center}
\begin{figure}[h]
\includegraphics[width=0.5\textwidth]{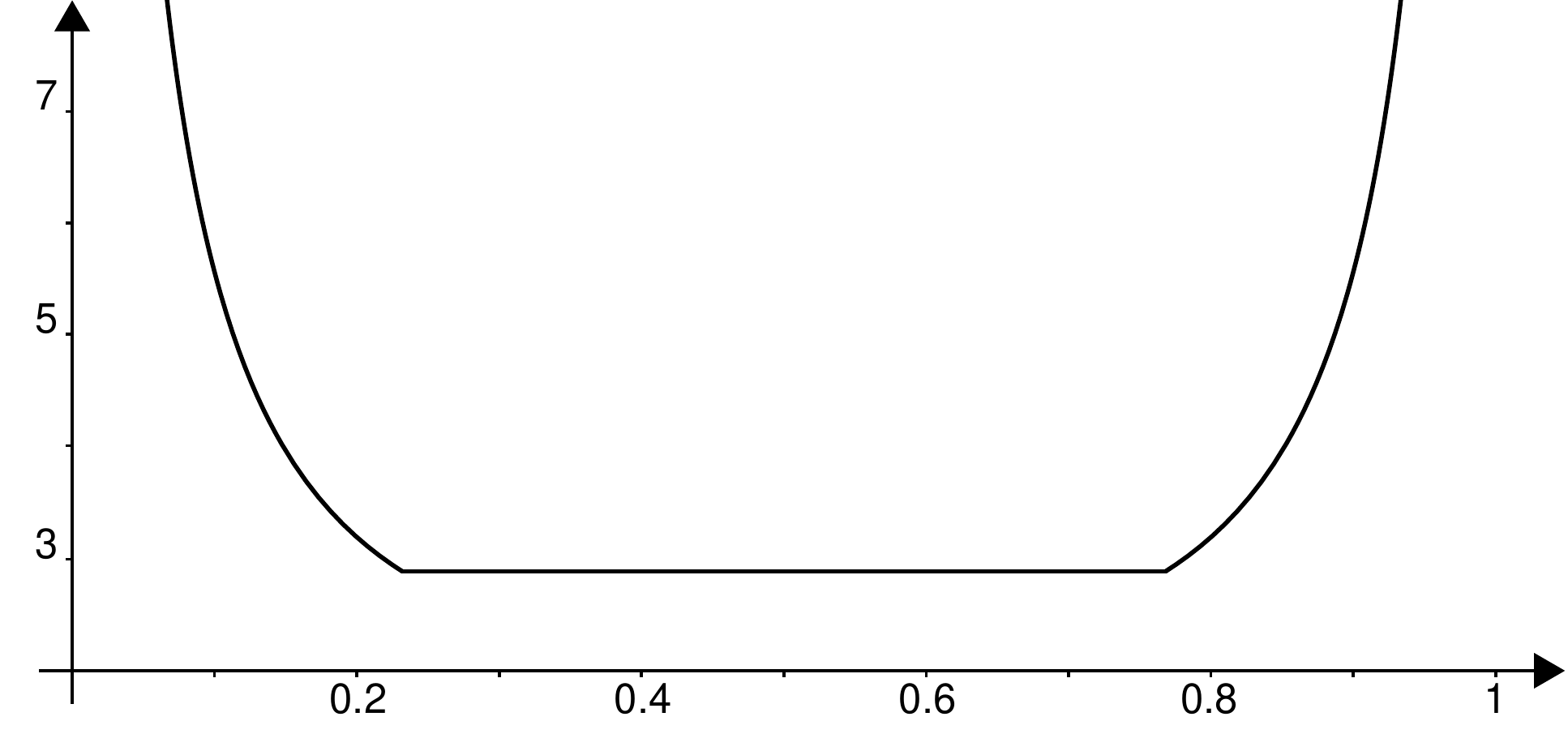}
\caption{Graph of $ - \lim_{n \to \infty} \frac{\log{ m_n(\omega;x,y)}}{\log n}$}
\label{fig:graph}
\end{figure}
\end{center}

\subsection{Finitely many Ruelle expanding maps.} \label{subsec:Finitely many Ruelle expanding maps}

In this paragraph, we describe a simple class of random dynamical systems to
which Theorem~\ref{thm:main} applies.

We begin with the description of the fibre maps. Let $(X,d)$ be a compact
metric space of bounded local complexity and recall that a
Lipschitz-continuous and surjective map $T: X \to X$ is \emph{Ruelle
expanding} if there exist $a>0$ and  $\lambda \in (0, 1)$, such that for any
$x, {y}, \tilde{x} \in X$ with $d(x, {y})<a$ and $T(\tilde{x})=x$, there
exists a unique $\tilde{y}\in X$ with $T(\tilde{y})={y}$ and $d(\tilde
x,\tilde y) < a$. Moreover, $d(\tilde{x}, \tilde{y}) \leq  \lambda d(x,y)$.
This class of maps was introduced by Ruelle
in~\cite{Ruelle--The-Thermodynamic-Formalism-For--CMP1989} and contains
subshifts of finite type as well as uniformly expanding maps on manifolds.

Now assume that $T_1, \dotsc, T_k: X \to X$ are Ruelle expanding maps which
are jointly mixing in the following sense. For any pair of nonempty open sets
$U,V \subset X$ we require that there exists $m \in \mathbb{N}$ such that
$(T_{i_n} \circ  \dotsm T_{i_1})^{-1}(U) \cap V \neq \emptyset$ for each
choice $i_1,\dotsc i_n \in \{1,\ldots k\}$ and $n> m$. In order to construct
the skew product, assume that $\Omega \subseteq \{1, \ldots k\}^\mathbb{Z}$
is a topologically mixing subshift of finite type and let
\[S : \Omega \times X \to \Omega \times X, ((\omega_i), x) \mapsto  (\sigma((\omega_i)),  T_{\omega_0}(x)),  \]
where $\sigma$ refers to the left shift. Here, it is worth noting that the
choice of $\Omega$ as a shift space is natural in the setting of finitely
many maps.

It remains to construct measures $\mathbb{P}$ and $\{ \mu_\omega : \omega \in
\Omega \}$ which satisfy the Lipschitz and mixing conditions in
Definitions~\ref{def:lipschitz_fibers},~\ref{def:mixing_fiber}
and~\ref{def:4mixing_base}. In order to do so, we first define a metric on
$\Omega$ by $d_s(\omega,\tilde\omega) \coloneqq s^{\min \{\abs{k} \st
\omega_k \neq \tilde\omega_k\}}$ for some fixed $s \in (0,1)$. Secondly,  we
fix Lipschitz continuous functions $\psi: \Omega \to \mathbb{R}$ and
$\varphi_i: X \to \mathbb{R}$ for $i=1, \ldots k$ and assume that
$\mathbb{P}$ is the unique Gibbs measure associated to $\psi$
(see~\cite{Bowen--Equilibrium-States-And-The--1975}). For the construction of
$\mu_\omega$, we proceed as follows. It is well known that the operator
defined by
\begin{equation}
\label{eq:def_transfer_operator}
\mathcal{L}_i(f)(x) \coloneqq \sum_{T_i(y)=x}
e^{\varphi_i(y)}f(y)
\end{equation}
acts on the space of Lipschitz functions. Furthermore,
as shown
in~\cite[Prop.~6.3]{Stadlbauer-Varandas-Zhang--Quenched-And-Annealed-Equilibrium--2020},
there exists $a> 0$ such that for any $\omega \in \Omega$, there exists a
probability measure $\mu_\omega$ such that for any Lipschitz continuous
function $f$ and $m,n > 0$,
\begin{equation} \label{eq:definition-of-mu_w-through-quotients}
\norm*{\frac{ \mathcal{L}_{\omega_m} \dotsm \mathcal{L}_{\omega_0}(f \,  \mathcal{L}_{\omega_{-1}} \dotsm \mathcal{L}_{\omega_{-n}}(\mathbf{1}) ) }
  {  \mathcal{L}_{\omega_m} \dotsm  \mathcal{L}_{\omega_{-n}}(\mathbf{1}) }  -  \int f\dd  \mu_\omega}_\infty \leq C e^{-a \min\{m,n\}} \cLip(f),
\end{equation}
where $\cLip(f)$ refers to the best Lipschitz constant of $f$.

\begin{prop} \label{prop:ruelle-expanding-semigroup}
Assume that $T_1, \ldots T_k$ are jointly mixing Ruelle expanding maps of the
compact metric space of bounded local complexity $X$ and that
$(\Omega,\sigma)$ is a two-sided, topologically mixing subshift of finite
type. Furthermore, assume that $\psi: \Omega \to \mathbb{R}$ and $\varphi_i:
X \to \mathbb{R}$ are Lipschitz continuous. Then the conclusions of
Theorem~\ref{thm:main} hold with respect to the equilibrium state
$\mathbb{P}$ of $\psi$ and $\{\mu_\omega : \omega \in \Omega\}$ as defined
in~\eqref{eq:definition-of-mu_w-through-quotients}.
\end{prop}
\begin{proof}
Observe that it follows from~\eqref{eq:definition-of-mu_w-through-quotients}
that  $\omega   \mapsto \mu_\omega$ is Lipschitz continuous with respect to
$d_t$, for $t \coloneqq \max \{s,e^{-a}  \}$. For ease of notation, set
$\mathcal{L}_{\omega_n \dotsm \omega_0}  \coloneqq \mathcal{L}_{\omega_n}
\dotsm \mathcal{L}_{\omega_0}$. Fix $x_0\in X$. For $f,g : X \to \mathbb{R}$
Lipschitz continuous and $k> 0$, it follows
from~\eqref{eq:definition-of-mu_w-through-quotients} that
\begin{align*}
\int f g \circ T^k_\omega & \dd\mu_\omega - \int f \dd\mu_\omega  \int g \dd\mu_{\theta^k\omega} \\
     & = \lim_{n \to \infty}   \frac{ \mathcal{L}_{\omega_n\ldots \omega_0}(f g \circ T^k_\omega \mathcal{L}_{\omega_{-1} \ldots \omega_{-n}}(\mathbf{1}))(x_0)}
       { \mathcal{L}_{\omega_n \ldots \omega_{-n}}(\mathbf{1})(x_0)}  -  \int f \dd\mu_\omega  \int g \dd\mu_{\theta^k\omega} \\
     & =  \lim_{n \to \infty}
      \frac{ \mathcal{L}_{\omega_n \dotsm \omega_k} \left( g  \mathcal{L}_{\omega_{k-1} \dotsm  \omega_{-n}}(\mathbf{1}) \left(
      \frac{\mathcal{L}_{\omega_{k-1} \dotsm \omega_0}(f  \mathcal{L}_{\omega_{-1} \dotsm \omega_{-n}}(\mathbf{1}))}{\mathcal{L}_{\omega_{k-1} \dotsm  \omega_{-n}}(\mathbf{1}) }  - \int f\dd\mu_\omega  \right)\right)(x_0)}
        {\mathcal{L}_{\omega_n \dotsm  \omega_{-n}} (\mathbf{1})(x_0)}  \\
      &\leq  C e^{-a (k-1)}  \cLip(f) \int \abs{g}  \dd\mu_{\theta^k\omega} \leq C e^{-a (k-1)} \normLip{f} \normLip{g}
\end{align*}
as $\normLip{ \, \cdot \, } = \norm{\cdot}_\infty + \cLip(\, \cdot \,)$.
Furthermore, by considering $f = \mathbf{1}$, it follows from the above
calculation that $\mu_{\omega} \circ T_\omega^{-1} = \mu_{\theta \omega}$.

The fact that $\mathbb{P}$ is exponential 4-mixing is standard. Let us
nevertheless explain the proof quickly, using the non-invertible, canonical
factor $(\Omega_+,\theta_+)$ of $(\Omega,\theta)$, where $\Omega_+ \coloneqq
\{ (\omega_i : i = 0, 1, \ldots): (\omega_i) \in \Omega \}$, $\theta_+$ is
the one-sided shift, $d_t^+$ is the usual shift metric with respect to the
parameter $t$ and $\pi: \Omega \to \Omega_+$ the canonical projection. It is
now crucial to recall some results
from~\cite{Bowen--Equilibrium-States-And-The--1975}. Firstly, we may assume
without loss of generality that $\psi = \psi_+ \circ \pi$ for a Lipschitz
function $\psi_+:\Omega_+\to \mathbb{R}$. The regularity of $\psi_+$ then
implies that the operator defined by
\[
P(f)(\omega) \coloneqq \sum_{\theta_+\tilde \omega = \omega} e^{\psi_+
(\tilde \omega)} f(\tilde \omega)\] acts on the space of Lipschitz continuous
functions on $\Omega_+$. Secondly, by adding a coboundary we may then assume
that $P(\mathbf{1}) = \mathbf{1}$ and  $\normLip{P^n(f) - \int f
\dd\mathbb{P}_+} \ll \lambda^n \cLip(f)$ for  $\mathbb{P}_+ \coloneqq
\mathbb{P}\circ \pi^{-1}$ and some $\lambda \in (0,1)$.

Now fix $m \in \mathbb{N}$ and choose for a given function $f : \Omega \to
\R$ a function $f^\ast: \Omega \to \R$ which on each ball of radius $t^{m+1}$
is constant and equal to some value of $f$ there. If $f$ is Lipschitz
continuous, then $\norm{f - f^\ast}_\infty \leq \cLip(f) t^m$ and there
exists $f^+:\Omega_+ \to \mathbb{R}$ such that $f^+ \circ \pi = f^\ast \circ
\theta^m$ and $\cLip(f^+) \leq t^{-m-1}\cLip(f)$.

Now assume that $f_1,f_2,g_1,g_2$ are Lipschitz continuous and that $0 \leq a
\leq b \leq c$, with $b-a \geq n \geq 0$ . Letting $m=\epsilon n$ for some
positive $\epsilon$ to be fixed below, we choose functions
$f^\ast_1,f^\ast_2,g^\ast_1,g^\ast_2$ as above. To check the $4$-mixing
property~\eqref{eq:4mixing}, we may replace $f_i$ and $g_i$ with $f_i^*$ and
$g_i^*$ respectively, as this introduces an exponentially small error.

Then the quantity to be estimated is
\begin{align*}
  &\abs*{\int (f^\ast_1    f^\ast_2\circ \theta^a   g^\ast_1\circ \theta^b   g^\ast_2 \circ \theta^c) \circ \theta^{m}  \dd\Pbb
       -  {\int f^\ast_1   f^\ast_2 \circ \theta^a \dd\Pbb}  {\int g^\ast_1   g^\ast_2 \circ \theta^{c-b}\dd\Pbb}} \\
  &= \abs*{\int  f^+_1    f^+_2\circ \theta_+^{a}   g^+_1 \circ \theta_+^b   g^+_2 \circ \theta_+^c  \dd\Pbb_+
       -  {\int f^+_1   f^+_2 \circ \theta_+^a \dd\Pbb_+}  {\int g^+_1   g^+_2 \circ \theta_+^{c-b}\dd\Pbb_+}} \\
  &= \abs*{\int   P^{b-a}(f^+_2 P^a(f^+_1))  g^+_1    g^+_2 \circ \theta_+^{c-b}  \dd\Pbb_+ - {\int f_2^+ P^a(f^+_1) \dd\Pbb_+} {\int g^+_1   g^+_2 \circ \theta_+^{c-b}\dd\Pbb_+}} \\
  & \ll \lambda^{b-a} \cLip(f_2^+ P^a(f^+_1)) \int \abs{g^+_1   g^+_2 \circ \theta_+^{c-b}} \dd\Pbb_+  \\
  & \ll \lambda^{b-a} t^{-2m} \normLip{f_1} \normLip{f_2} \norm{g_1}_\infty \norm{g_2}_\infty.
\end{align*}
The prefactor is bounded by $\lambda^n t^{-2\epsilon n}$. If $\epsilon$ is
small enough, it is exponentially small as desired.
\end{proof}

\begin{rmk}
If $T_1, \dotsc, T_k$ are Ruelle expanding maps defined on a connected and
compact Riemannian manifold $X$, then the semigroup generated by these maps
is always jointly mixing
(see~\cite[Prop.~3.3]{Stadlbauer-Varandas-Zhang--Quenched-And-Annealed-Equilibrium--2020}).
Furthermore, as Riemannian manifolds are always of bounded local complexity,
the conclusions of Proposition~\ref{prop:ruelle-expanding-semigroup} hold
without these two hypotheses.
\end{rmk}

%

\subsection{Non-uniformly expanding local diffeomorphisms}
\label{subsec:Non-uniformly expanding local diffeomorphisms} We now give an
example in which the fiber maps are nonuniformly expanding: contrary to
Paragraph~\ref{subsec:Finitely many Ruelle expanding maps}, there may be some
region where the maps are contracting, but the expansion is still winning on
average. We adapt the setting
in~\cite{Stadlbauer-Suzuki-Varandas--Thermodynamic-Formalism-For-Random--CMP2021}.
In contrast to the situation in there, we have to assume that the base
transformation has the following mixing property.
\begin{enumerate}
 \item[(H0)] Assume that $(\Omega,d)$ is a compact metric space, that
     $\theta: \Omega \to \Omega$ is a bi-Lipschitz homeomorphism and that
     $\Pbb$ is a $\theta$-invariant probability measure with stretched
     exponential 4-mixing.
\end{enumerate}
Furthermore, let $X$ be a compact connected Riemannian manifold and let
$\{T_\omega\}_{\omega \in \Omega}$ be a family of $C^1$-maps on $X$ with the
following properties.
\begin{enumerate}
\item[(H1)] For each $\omega$, the map $T_\omega$ is a surjective, local
    diffeomorphism.
\item[(H2)] There exists $\delta>0$ such that for every $(\omega,x)\in X$,
    there exists an open neighborhood $U^\omega_x$ of $x$  with
    $T_\omega\restr{U^\omega_x}:U^\omega_x \to B(T_\omega(x), \delta)$
    invertible,
\item[(H3)] There exists $C> 0$ such that $\norm{DT_\omega(x)} \leq C$ for
    all $\omega \in \Omega$ and all $x\in X$.
\end{enumerate}
Observe that the connectedness  and compactness of $X$ imply that the degree
$\deg(T_\omega)$ of $T_\omega$, i.e., the number of preimages of $T_\omega$
is finite and constant for each $\omega$. Moreover, (H3) ensures that this
degree is uniformly bounded.

We also assume the following geometric conditions. There are  random
variables $\sigma_\omega > 1, L_\omega \geq 0$ and $0\leq p_\omega,
q_\omega<\deg(f_\omega) $ with $L\coloneqq \sup L_\omega < \infty$ such that
for each $\omega$,
\begin{enumerate}
\item[(H4)] there  exists a covering $\mathcal{P}_\omega = \{P^\omega_1,
    \dotsc,P^\omega_{p_\omega}, \dotsc, P^\omega_{p_\omega+q_\omega}\}$ of
    $X_\omega$ such that every $T_\omega\restr{P_i}$ is injective,
    $\norm{DT_\omega(x)^{-1}} \leq \sigma_\omega^{-1} < 1$ for each $x\in
    P^\omega_{1} \cup \dotsb \cup P^\omega_{p_\omega}$, and
    $\norm{DT_\omega(x)^{-1}}\leq L_\omega$ for every $x\in X$,
\item[(H5)]
\[\sup  \left\{ \log\frac{\sigma_\omega^{-1}p_\omega+L_\omega q_\omega}{\deg(f_\omega)} : \omega \in \Omega \right\}< 0 .\]
\end{enumerate}
Finally, assume that  $\{\phi_\omega\}$ is a family of real valued functions
in $C^1(X)$, referred to as potentials, such that $\sup \{
\norm{D\phi_\omega}_\infty :  \omega \in \Omega\} < \infty$ and for all
$\omega$
\begin{enumerate}
\item[(H6)] $\displaystyle  \sup_{x \in X_\omega}\phi_\omega-\inf_{x \in
    X_\omega} \phi_\omega  + \log \pare*{1+\norm{D\phi_\omega}_\infty \diam
    X}  < - \log \frac{\sigma_\omega^{-1}p_\omega+L_\omega
    q_\omega}{\deg(f_\omega)}.$
\end{enumerate}
We now motivate conditions (H4--6). (H4) states that regions of contraction
and expansion may coexist whereas (H5) implies that expansion dominates the
contraction. Finally, (H6) can be seen as an upper bound of the global and
local oscillation by the combinatorial expansion in (H4). However, in
contrast to the bounds in average
in~\cite{Stadlbauer-Suzuki-Varandas--Thermodynamic-Formalism-For-Random--CMP2021},
we have to ask for uniform bounds in (H5) and (H6) due to the uniform
hypothesis in Theorem~\ref{thm:main}.

Furthermore, in order to guarantee the continuous variation with respect to $\omega$, we assume that $T$ and $\varphi$ vary Lipschitz continuously with respect to $\omega$, that is we require that
\begin{enumerate}
\item[(H7)] there exists $C_T> 0$ with $d(T_\omega(x), T_{\omega'}(x)) \leq
    C_T d(\omega,\omega')$ for all $\omega,\omega' \in \Omega$ and all
    $x\in X$,
\item [(H8)] there exists $C_\varphi> 0$ with $\abs*{ \varphi_\omega(x) -
    \varphi_{\omega'}(x)} \leq C_\varphi d(\omega,\omega')$ for all
    $\omega,\omega' \in \Omega$ and all $x\in X$.
\end{enumerate}

The relevant measures $\{\mu_\omega\}$ are now given by application of
Theorem A
in~\cite{Stadlbauer-Suzuki-Varandas--Thermodynamic-Formalism-For-Random--CMP2021}:
denoting by $\mathcal{L}_\omega$ the transfer operator associated to
$T_\omega$ and $\phi_\omega$ as in~\eqref{eq:def_transfer_operator}, there
exist families of positive constants $\{\lambda_\omega\}$, of differentiable
functions $\{h_\omega\}$ and probability measures $\nu_\omega$ such that
$\mathcal{L}_\omega(h_\omega) = \lambda_\omega h_{\theta\omega}$ and
$\mathcal{L}_\omega^\ast(\nu_{\theta\omega}) = \lambda_\omega \nu_\omega$.
Furthermore, for $\dd\mu_\omega \coloneqq h_\omega \dd\nu_\omega$, it follows
that $\mu_\omega = \mu_{\theta \omega} \circ T_\omega^{-1}$ and that a
fibered exponential decay of correlation for $C^1$-observables holds
(Corollary~1
in~\cite{Stadlbauer-Suzuki-Varandas--Thermodynamic-Formalism-For-Random--CMP2021}).
Finally, we would like to remark that the standard examples for this class
are random Manneville-Pomeau maps or sufficiently random perturbations of
non-uniformly expanding maps with respect to a potential sufficiently close
to zero (see Examples~3.1 and~3.2
in~\cite{Stadlbauer-Suzuki-Varandas--Thermodynamic-Formalism-For-Random--CMP2021}
for more details).

\begin{prop}
  \label{prop:random-local-diffeo}
Assume that (H0--H8) hold. Then there exists $\alpha>0$ such that the
assumptions of Theorem~\ref{thm:main} hold with respect to $\Pbb$, the metric
$d_\alpha$ defined by $d_\alpha(\omega,\omega') \coloneqq
d(\omega,\omega')^\alpha$ on $\Omega$ and $\{\mu_\omega : \omega \in
\Omega\}$ as defined above.
\end{prop}

\begin{proof} Before starting with the proof, we remark that we will write $a \ll b$ whenever there exists  $C> 0$, depending exclusively on (H0--H8), such that  $   a \leq  C b$.

In order to verify the assumptions of Theorem~\ref{thm:main}, we have to
modify the construction
in~\cite{Stadlbauer-Suzuki-Varandas--Thermodynamic-Formalism-For-Random--CMP2021}
slightly. In order to do so, recall that it is shown there that there exists
a positive and almost surely finite random variable $\kappa_\omega$ such that
the family of cones
\begin{equation} \label{eq:differentiable-cones}
\Lambda_{\omega}
	\coloneqq \left\{g_\omega\in C^1(X) : g_\omega>0 \text{ and }  \norm{Dg_\omega}_\infty \leq \kappa_\omega \inf_{x \in X} g_\omega(x) \right\}
\end{equation}
satisfies $\mathcal{L}_\omega(\Lambda_{\omega}) \subset
\Lambda_{\theta\omega}$ a.s. However, as (H6) provides a  uniform bound, it
follows from the construction (see Formula 5.2
in~\cite{Stadlbauer-Suzuki-Varandas--Thermodynamic-Formalism-For-Random--CMP2021})
that $\kappa_\omega$ is in fact uniformly bounded. Hence
(see~\cite[Claim~1]{Stadlbauer-Suzuki-Varandas--Thermodynamic-Formalism-For-Random--CMP2021}),
there exist $C >0$ and $\vartheta \in (0,1)$ such that,
 for every $m,n \geq 1$ and every $\varphi \in\Lambda_{\theta^{-(n+m)}(\omega)}$ and  $\psi\in\Lambda_{\theta^{-n}\omega}$,
\begin{equation}\label{eq:geometric-cauchy-sequence}
\Theta_{\omega} \pare*{\mathcal{L}_{\theta^{-(n+m)}\omega}^{n+m}
\varphi, \mathcal{L}_{\theta^{-n}\omega}^{n}\psi} \leq C  \, \vartheta^n \; \Theta_{\theta^{-n}\omega} \pare*{\mathcal{L}_{\theta^{-(m+n)}\omega}^{m}\varphi, \psi} .
\end{equation}
In here, $\Theta_{\omega}$ refers to the Hilbert metric on  $\Lambda_\omega$
and $\mathcal{L}_\omega^n$ stands for $\mathcal{L}_{\theta^{n-1}\omega}
\dotsm \mathcal{L}_\omega$ which we now analyze briefly. Firstly, by dividing
the defining relation by $g_\omega$  in~\eqref{eq:differentiable-cones}, one
obtains that $\norm{D \log g_\omega}_\infty \leq \kappa_\omega$ and, in
particular, that $q_\omega(x)/q_\omega(y)\leq \kappa_\omega \diam(X)$ for all
$x,y \in X$. Secondly, if $\Theta_\omega(f,g) < \epsilon$ for some
$\epsilon>0$, then there are $t'\geq t > 0$ with $\log t' - \log t <
\epsilon$ such that $0 \leq \kappa_\omega  \inf (f - tg)$ and $0 \leq
\kappa_\omega \inf (t'g  - f)$. Hence,  $t \leq f/g \leq t' $ and  $0 \leq
f/g - t <  \epsilon$.


\subsubsection*{Construction of $\nu_\omega$}
Assume that $f \in \Lambda_\omega$. Then~\eqref{eq:geometric-cauchy-sequence}
implies, for $\epsilon_n\coloneqq C  \vartheta^n \Theta_{\omega}(f,
\mathbf{1})$, that
$\Theta_{\theta^n\omega}(\mathcal{L}_\omega^n(f),\mathcal{L}_\omega^n(\mathbf{1}))
\leq \epsilon_n$. With $r_n \coloneqq \inf_x
\mathcal{L}_\omega^n(f)(x)/\mathcal{L}_\omega^n(\mathbf{1})(x)$, one then
obtains from the above that $r_n \leq
{\mathcal{L}_\omega^n(f)}/{\mathcal{L}_\omega^n(\mathbf{1})} \leq r_n
e^{\epsilon_n}$. Hence, as
\[
r_{m+n} \leq
\frac{\mathcal{L}_\omega^m(\mathcal{L}_{\theta^m\omega}^n(f))}{\mathcal{L}_\omega^{m+n}(\mathbf{1})}
\leq r_n e^{\epsilon_n} \hbox{ and } r_{m+n}e^{\epsilon_{m+n}} \geq
\frac{\mathcal{L}_\omega^m(\mathcal{L}_{\theta^m\omega}^n(f))}{\mathcal{L}_\omega^{m+n}(\mathbf{1})}
\geq r_n,\] $(\log r_n)$ is a Cauchy sequence and, in particular, $\lim_n r_n
\in (0,\infty)$. Moreover, as $\Theta_{\omega}(f, \mathbf{1})$ is uniformly
bounded, the function
${\mathcal{L}_\omega^n(f)}/{\mathcal{L}_\omega^n(\mathbf{1})}$ converges
uniformly to a constant $\nu_\omega(f)$. Moreover, this constant satisfies
\begin{equation} \label{eq:log-exp-decay-of measures}
 \norm*{\log \nu_\omega(f) - \log \frac{\mathcal{L}_\omega^n(f)}{\mathcal{L}_\omega^n(\mathbf{1})}}_\infty \ll \vartheta^n \;\;  \forall f \in \Lambda_\omega, n \in \N, \omega \in \Omega.
\end{equation}
We now extend the domain of $\nu_\omega$ to Lipschitz functions, using the
following standard fact: There exists $C>0$ such that for any $\epsilon > 0$
and Lipschitz function $f : X \to \R$, there is $f^\ast \in C^1(X)$ with
$\norm{f - f^\ast}_\infty \leq \epsilon$ and $\norm{Df^\ast} \leq C
\cLip(f)$. Let us recall how this fact is proved, using \emph{mollifiers} as
follows. Assume that $\ell:  \R^d \to [0,\infty)$ is a function in $C^1$ such
that $\ell$ is supported on $\{x : \norm{x}\leq 1 \}$ and $\int \ell \dd x
=1$. For $\ell_\epsilon(x) \coloneqq \epsilon^{-d} \ell(x/\epsilon)$, the
convolution $g \ast \ell_\epsilon$ with a Lipschitz function $g$ is in
$C^1(X)$ and $\norm{g -  g \ast \ell_\epsilon}_\infty \leq 2\cLip(g)
\epsilon$. If, in addition, $\ell(x) = j(\norm{x})$ for some $j:[0,1] \to \R$
(i.e., $\ell$ is constant on spheres), a straightforward calculation shows
that $\norm{D(g \ast \ell_\epsilon)} \ll \cLip(g)$ where the implicit
constant in `$\ll$' only depends on $\ell$. Finally, by employing an argument
based on a partition of unity, one obtains that the same holds for the
Riemannian manifold $X$.

In order to employ~\eqref{eq:log-exp-decay-of measures}, note that $f^\ast +
c \in \Lambda_ \omega$ for $c =  \norm{Df^\ast}_\infty/\kappa_\omega - \inf
f^\ast$. Hence,
\begin{align*}
  \norm*{\nu_\omega(f^\ast) -  \frac{\mathcal{L}_\omega^n(f)}{\mathcal{L}_\omega^n(\mathbf{1})}}_\infty  -  \epsilon  &
  \leq   \norm*{\nu_\omega(f^\ast+c) -  \frac{\mathcal{L}_\omega^n(f^\ast+c)}{\mathcal{L}_\omega^n(\mathbf{1})}}_\infty  \\
 &\ll  \vartheta^n \nu_\omega(f^\ast+c)   \ll  \vartheta^n \pare*{\norm{Df^\ast}_\infty/\kappa_\omega +  \norm{f^\ast - \inf f^\ast}_\infty}\\
 &\ll  \vartheta^n \cLip(f) \pare*{\kappa_\omega^{-1} + \diam (X)} \ll
  \vartheta^n \cLip(f).
 \end{align*}
As $\epsilon>0$ is arbitrary, the definition of $\nu_\omega$ extends to Lipschitz continuous functions and
\begin{equation} \label{eq:exp-decay-of-conformal-measures}
\norm*{ {\mathcal{L}_\omega^n(f)}/{\mathcal{L}_\omega^n(\mathbf{1})} - \nu_\omega(f)    }_\infty \ll \vartheta^n \cLip(f).
\end{equation}

\subsubsection*{Construction of $h_\omega$}
Let $\lambda_{\theta^{-n}\omega}^n \coloneqq
\nu_{\omega}(\mathcal{L}_{\theta^{-n}\omega}^n (\mathbf{1}))$ and
$h_\omega^n\coloneqq \mathcal{L}_{\theta^{-n}\omega}^n
(\mathbf{1})/\lambda_{\theta^{-n}\omega}^n$. As $\Theta_\omega$ is a
projective metric,  \eqref{eq:geometric-cauchy-sequence} implies that
$(\mathcal{L}_{\theta^{-n}\omega}^n (\mathbf{1}))$ and $(h_n)$ are Cauchy
sequences with respect to $\Theta_\omega$. Hence, for $r_{n,k} \coloneqq
\inf_x h^{n+k}_\omega(x) /  h^n_\omega(x)$, it follows in analogy to the
construction of $\nu_\omega$ that
\[
r_{n,k} \leq {h^{n+k}_\omega}/{h^n_\omega} \leq r_{n,k} e^{\epsilon_n}.
\]
Moreover, by multiplying with $h^n_\omega$ and integrating with respect to
$\nu_\omega$, it follows that $ \abs{\log r_{n,k}}\leq \epsilon_n$. Hence,
$h_\omega = \lim h^n_\omega$ exists and satisfies
\begin{equation} \label{eq:exp-convergence-to-h}
\norm*{\log h^n_\omega - \log h_\omega}_\infty \ll \vartheta^n.
\end{equation}

\subsubsection*{Exponential mixing}
We now show that $\dd\mu_\omega = h_\omega \dd\nu_\omega$ mixes exponentially
along fibers with respect to Lipschitz functions
(in~\cite{Stadlbauer-Suzuki-Varandas--Thermodynamic-Formalism-For-Random--CMP2021},
it is only shown for functions in $C^1(X)$). However, it follows
from~\eqref{eq:exp-decay-of-conformal-measures}
and~\eqref{eq:exp-convergence-to-h} that, for $f$ Lipschitz and some uniform
$C> 0$,
\begin{align*}
 \norm*{\log \nu_\omega(f h_\omega) -
 \log \frac{\mathcal{L}_\omega^n(f h_\omega)}{\lambda^n_\omega h_{\theta^n\omega}} }_\infty
 &\leq  \norm*{\log \nu_\omega(f h_\omega) -
 \log \frac{\mathcal{L}_\omega^n(f h_\omega)}{\mathcal{L}_\omega^n(\mathbf{1})}}_\infty  +
\norm*{\log \frac{ \lambda^n_\omega h_{\theta^n\omega} }{\mathcal{L}_\omega^n(\mathbf{1})} }_\infty \\
&\leq  C \vartheta^n +
\norm*{\log h_{\theta^n\omega} - {\mathcal{L}_\omega^n(\mathbf{1})}/{ \lambda^n_\omega } }_\infty  \ll \vartheta^n.
\end{align*}
 Observe that $\tilde{\mathcal{L}}_\omega: f \mapsto  \mathcal{L}_\omega^n(f h_\omega) / (\lambda^n_\omega h_{\theta^n\omega})$ is the transfer operator
 of $T_\omega$ with respect to $\mu_\omega$. Hence, as $\cLip(h_\omega)$ is uniformly bounded, it follows that   $ \norm{\tilde{\mathcal{L}}^n_\omega(f) - \mu_\omega(f)}_\infty \ll \vartheta^n \cLip(f)$. The exponential mixing along fibers follows from this.

\subsubsection*{H\"older continuity of $\mu_\omega$}
Note that Kantorovich's duality for the Wasserstein metric $W$ on probability
measures implies that the condition in Definition~\ref{def:lipschitz_fibers}
is equivalent to Lipschitz continuity with respect to $W$. So assume that $f$
is a Lipschitz function on $X$ with  $\cLip(f)\leq 1$ and $\inf_{x \in
X_\omega} f(x) =0$. Then $\norm*{f}_\infty \leq \cLip(f) \diam(X) $ and
\begin{align*}
 \int f \dd (\mu_\omega - \mu_{\omega'})
 & =   \int f (h_\omega - h_{\omega'}) \dd  \nu_\omega  + \int f h_{\omega'} \dd  (\nu_\omega  -  \nu_{\omega'}) \\
 & \leq  \norm*{ h_\omega - h_{\omega'}}_\infty \int f \dd\nu_\omega + \cLip(fh_{\omega'}) W(\nu_\omega,\nu_{\omega'}) \\
 & \leq  \diam(X) \norm*{ h_\omega - h_{\omega'}}_\infty +
\pare*{\cLip(h_{\omega'}) \diam(X) + \norm*{ h_{\omega'}}_\infty} W(\nu_\omega,\nu_{\omega'}) \\
&\ll     \norm*{ h_\omega - h_{\omega'}}_\infty + \normLip*{h_{\omega'}} W(\nu_\omega,\nu_{\omega'}).
\end{align*}
As $h_{\omega'} \in \Lambda_{\omega'}$ and $\int h_{\omega'} \dd
\nu_{\omega'} =1$, the Lipschitz norm $\normLip*{h_\omega}$ is uniformly
bounded. So it remains to show that $\omega \to h_\omega$ is H\"older with
respect to the sup-norm and that $\nu_\omega$ is H\"older  with respect to $W$,
which essentially is a corollary of the above and the following estimate.

We recall that constants $\delta$, $L$ and $C_T$ have been introduced
respectively in (H2), (H4) and (H7).

\begin{lem}
\label{lemma:continuity-of-Ln} There exists $A\geq 1$ such that for $n \in \N$,  $\omega, \omega' \in
\Omega$ with $A^n d(\omega,  \omega') < 1$ (or
$d(\theta^n\omega,  \theta^n\omega') A^n <1$, respectively) and
$f \in \Lambda_\omega \cap  \Lambda_{\omega'}$,
 \[  \norm*{ \log \mathcal{L}_\omega^n(f) - \log \mathcal{L}_{\omega'}^n(f)}_\infty
\ll \begin{cases}
        A^n d(\omega,\omega') & \quad \text{if } d(\omega,  \omega') A^n < 1, \\
        A^n d(\theta^n\omega,\theta^n\omega') &\quad\text{if }  d(\theta^n\omega,  \theta^n\omega') A^n < 1.
    \end{cases}
 \]
\end{lem}
\begin{proof}
The estimate relies on the control of the distances between the preimages of
$T^k_{\omega}$ and $T^k_{\omega'}$, where $k = 1, \ldots, n$ and $\omega$ and $\omega'$ are sufficiently close.

So assume that $x,y,y' \in X$ and $\omega,\omega'\in \Omega$ with
$T_\omega(x)=y$,  $d(y,y')<\delta$,  and $d(\omega,\omega') < C_T^{-1}
\delta$. As $d(y,y')<\delta$, $z:= (T_{\omega}\restr{U^{\omega}_x})^{-1}(y')$
is well defined and $d(x,z) \le L d(y,y')$ by (H4). It now follows from (H7)
that $d(T_{\omega}(z),T_{\omega'}(z))< \delta$. Hence, $x' :=
(T_{\omega'}\restr{U^{\omega'}_z})^{-1}(y')$ is also well defined and
\begin{equation}
\label{eq:Lipschitz-constants-for-inverses-v3}
\begin{split}
d(x,x') & \leq d(x, z) + d(z, x')
\leq L d(T_\omega x, T_\omega z) + M d(T_{\omega'} z, T_{\omega'}x')
\\& 
= L d(y,y') + L d(T_{\omega'}(z), T_\omega z) \leq  L d(y,y') + LC_T d(\omega',\omega).
\end{split}
\end{equation}
Note that (H3) implies that $T_{\omega'}\restr{B(x,C^{-1} \delta)}$ is injective. In particular,
there is at most one $x'\in X$ with $T_{\omega'}(x') =y'$ and $d(x,x') < C^{-1}\delta$. Hence, if $d(x,x') < C^{-1}\delta$, then  $x'$ is unique.

Now assume that $n \in \N$ and $x,y\in X$ with $T^n_\omega(x)=y$ are given. By iterating the above construction of preimages, one then obtains a unique $x'\in X$ with $T^n_{\omega'}(x')=y$ whose $\omega'$-orbit stays close to the $\omega$-orbit of $x$ until time $n$, provided that $d(\omega,\omega')$ is sufficiently small. Namely, it easily follows by induction that this happens whenever $d(\theta^k\omega, \theta^k\omega') < (CC_T)^{-1} \delta (2L )^{-k-1}$ for $k = 0, \ldots n-1$. As $d(\theta^k\omega, \theta^k\omega') \leq \cLip(\theta)^k d(\omega,\omega')$, there exists $A\geq 1$ such that $\omega$ and $\omega'$ are sufficiently close if $ d(\omega,\omega') <  A^{-n}$.

Furthermore, if $ d(\omega,\omega') <  A^{-n}$, it follows from (H8), (H3) and again by induction from \eqref{eq:Lipschitz-constants-for-inverses-v3} and eventually enlarging $A$ that
\begin{equation}
\label{eq:contiunuity-of-Birkhoff-sums-wrt-omega}
\begin{split}
      & \abs*{ \sum_{k=0}^{n-1}  \varphi_{\theta^k \omega} (T^k_\omega(x)) - \varphi_{\theta^k \omega'} (T^k_{\omega'}(x'))}
  \leq \sum_{k=0}^{n-1}  C_\varphi d(\theta^k\omega, \theta^k\omega') +  \sum_{k=0}^{n-1} Cd(T^k_{\omega}(x),T^k_{\omega'}(x')) \\
  \leq & C_\varphi \sum_{k=0}^{n-1} \cLip(\theta)^k d(\omega, \omega')+  CC_T \sum_{k=0}^{n-1} \sum_{\ell = k}^{n-1} L^{n-\ell} d(\theta^\ell\omega, \theta^\ell \omega') \\
  \ll  & A^n d(\omega, \omega').
\end{split}
\end{equation}
As there is a one-to-one relation between the preimages of $T^n_\omega$ and $T^n_{\omega'}$,
it follows from \eqref{eq:contiunuity-of-Birkhoff-sums-wrt-omega} by a standard argument, for $f \in \Lambda_{\omega} \cap
\Lambda_{\omega'}$, that
\begin{align*}
\abs*{\mathcal{L}_\omega^n(f)(x) - \mathcal{L}_{\omega'}^n(f)(x)}
& \ll  A^n d(\omega,\omega') \pare*{ \mathcal{L}_\omega^n(f)(x) +   \cLip(\log f) \mathcal{L}_{\omega'}^n(f)(x) } \\
&\ll A^n d(\omega,\omega') \pare*{ \mathcal{L}_\omega^n(f)(x) +    \mathcal{L}_{\omega'}^n(f)(x) },
\end{align*}
where we have used that $\cLip(\log f)$ is uniformly bounded as $f \in
\Lambda_{\omega'}$. Furthermore, observe
that~\eqref{eq:contiunuity-of-Birkhoff-sums-wrt-omega}, $d(\omega,  \omega')
A^n <  1$ and the uniform bound on $\cLip(\log f)$ imply that
$\mathcal{L}_\omega^n(f)(x) \ll \mathcal{L}_{\omega'}^n(f)(x) \ll
\mathcal{L}_\omega^n(f)(x)$. Hence, the first assertion of the lemma follows
by dividing the last estimate by  $\mathcal{L}_\omega^n(f)(x)$.
The second part of Lemma~\ref{lemma:continuity-of-Ln} follows by
precisely the same arguments.
\end{proof}

By combining Lemma~\ref{lemma:continuity-of-Ln} with
estimate~\eqref{eq:log-exp-decay-of measures}, one obtains  for $\omega,
\omega'$ sufficiently close, $x\in X$ and $f \in \Lambda_\omega \cap
\Lambda_{\omega'}$ and $\nu^n_\omega(f) \coloneqq
{\mathcal{L}_\omega^n(f)(x)}/{\mathcal{L}_\omega^n(\mathbf{1})(x)}$ that
\begin{equation}
\label{eq:log-estimate-for-nu_omega-wrt-to-the cone}
\begin{split}
\abs{\log \nu_\omega(f) - \log \nu_{\omega'}(f)} &\leq
\abs{\log \nu_\omega(f) - \log \nu^n_{\omega}(f)} + \abs{\log \mathcal{L}_\omega^n(f)(x) - \log  \mathcal{L}_{\omega'}^n(f)(x)  } \\
& \quad + \abs{\log \mathcal{L}_\omega^n(\mathbf{1})(x) - \log  \mathcal{L}_{\omega'}^n(\mathbf{1})(x)  } +
\abs{\log \nu^n_{\omega'}(f) - \log \nu_{\omega'}(f)}\\
&\ll \vartheta^n + A^n d(\omega,  \omega').
\end{split}
\end{equation}
For $t$ given by $d(\omega,  \omega') = A^{-t}$, $n \coloneqq \lfloor t \log
A /(\log A - \log \vartheta) \rfloor$, it then follows that
\[
\abs{\log \nu_\omega(f) - \log \nu_{\omega'}(f)} \ll d(\omega,  \omega')^{\frac{- \log \vartheta}{\log A - \log \vartheta}} = d(\omega,  \omega')^\alpha,
\]
where  $\alpha \coloneqq {- \log \vartheta}/\pare{\log A - \log \vartheta}$.
By repeating the approximation argument
in~\eqref{eq:exp-decay-of-conformal-measures}, one obtains that  $
\abs{\nu_\omega(f) -   \nu_{\omega'}(f)} \ll \cLip(f) d(\omega,
\omega')^\alpha$ and by Kantorovich's duality that $W(\nu_\omega,
\nu_{\omega'}) \ll d(\omega,  \omega')^\alpha$.
With respect to $h_\omega$, estimate~\eqref{eq:exp-convergence-to-h},
Lemma~\ref{lemma:continuity-of-Ln},
estimate~\eqref{eq:log-estimate-for-nu_omega-wrt-to-the cone}, a further
application of Lemma~\ref{lemma:continuity-of-Ln}
and~\eqref{eq:exp-convergence-to-h} (in this order) imply that
\begin{align*}
   \abs{\log h_\omega(x)  - \log h_{\omega'}(x)}
 & \leq
    \abs*{\log \frac{h_\omega(x)}{h^n_{\omega}(x)}}
  + \abs*{\log \frac{\mathcal{L}_{\theta^{-n}\omega}^n(\mathbf{1})(x)}{  \mathcal{L}_{\theta^{-n}\omega'}^n(\mathbf{1})(x)}  }
  + \abs*{ \log \frac{\nu_\omega  (\mathcal{L}_{\theta^{-n}\omega}^n(\mathbf{1})) }{ \nu_{\omega'} (\mathcal{L}_{\theta^{-n}\omega}^n(\mathbf{1})) }} \\
&  \quad +
 \abs*{ \log \frac{\nu_{\omega'}  (\mathcal{L}_{\theta^{-n}\omega}^n(\mathbf{1})) }{ \nu_{\omega'} (\mathcal{L}_{\theta^{-n}\omega'}^n(\mathbf{1}))} } +     \abs*{\log \frac{h_{\omega'}(x)}{h^n_{\omega'}(x)}}    \\
 & \ll  \vartheta^n + A^n d(\omega,  \omega') + d(\omega,  \omega')^\alpha + A^n d(\omega,  \omega') +   \vartheta^n.
 \end{align*}
With respect to the same choice of $t$ and $n$ as above, it follows that
$\norm{h_\omega  - h_{\omega'}}_\infty \ll d(\omega,  \omega')^\alpha$. This
concludes the proof of Proposition~\ref{prop:random-local-diffeo}, except for
the fact that $\theta$ is exponentially $4$-mixing with respect to H\"older
functions, which is proved in the next lemma.
\end{proof}

\begin{lem}
\label{lem:extend_Lip_Hol} Let $\theta : \Omega \to \Omega$ be a map on a
compact metric space which is exponentially $4$-mixing for Lipschitz
functions. Then it is also exponentially $4$-mixing with respect to
H\"older-continuous functions of any given positive exponent.
\end{lem}
This fact is not specific to exponential $4$-mixing: it works for any kind of
mixing rate. It follows from a standard interpolation argument, approximating
H\"older-continuous functions with Lipschitz ones. This is a standard fact on
Riemannian manifolds using mollifiers as we discussed above in the
construction of $\nu_\omega$, but it works in any metric space as we explain
now.

\begin{lem}
Let $(\Omega, d)$ be a metric space, and $\alpha>0$, $\kappa>0$. For any
$\alpha$-H\"older-continuous $f:\Omega \to \R$, there exists a Lipschitz
function $f^* : \Omega \to \R$ with $\norm{f-f^*}_\infty \leq 2
\mathrm{Hol}_\alpha(f) \kappa^\alpha$ and $\cLip(f^*) \leq
\kappa^{-(1-\alpha)} \mathrm{Hol}_\alpha(f)$, where
\begin{equation*}
  \mathrm{Hol}_\alpha(f) = \sup_{x \ne y} \frac{\abs{f(x) - f(y)}}{d(x, y)^\alpha}
\end{equation*}
is the best $\alpha$-H\"older constant of $f$.
\end{lem}
\begin{proof}
Let $A$ be a maximal $\kappa$-separated set in $\Omega$. Let also $M =
\mathrm{Hol}_\alpha(f) \kappa^{-(1-\alpha)}$. The restriction of $f$ to $A$
satisfies, for $x\ne y$, the inequality
\begin{equation*}
  \abs{f(x)-f(y)} \leq \mathrm{Hol}_\alpha(f) d(x,y)^\alpha \leq M d(x,y),
\end{equation*}
as $d(x,y) \geq \kappa$. Define $f^*$ on $\Omega$ by
\begin{equation*}
  f^*(x) = \inf_{y\in A} f(y) + M d(x,y).
\end{equation*}
The previous inequality ensures that this function coincides with $f$ on $A$,
and that it is $M$-Lipschitz globally. Let us check that $\norm{f-f^*}_\infty
\leq 2 \mathrm{Hol}_\alpha(f) \kappa^\alpha$. Take $x \in \Omega$. By
maximality of $A$, there exists $y\in A$ with $d(x,y)\leq \kappa$. Then, as
$f(y)=f^*(y)$, we get
\begin{equation*}
  \abs{f(x)-f^* (x)} \leq \abs{f(x) - f(y)} + \abs{f^*(x)-f^* (y)}
  \leq \mathrm{Hol}_\alpha(f) \kappa^\alpha + \cLip(f^*) \kappa
  \leq 2 \mathrm{Hol}_\alpha(f) \kappa^\alpha.
\qedhere
\end{equation*}
\end{proof}

\begin{proof}[Proof of Lemma~\ref{lem:extend_Lip_Hol}]
Start from $\alpha$-H\"older continuous functions $f_1, f_2, g_1, g_2$ and $a
\leq b \leq c$ for which one wants to prove~\eqref{eq:4mixing} with the
Lipschitz norm replaced by the H\"older norm. Let $\kappa= e^{-\epsilon n}$
with $\epsilon$ suitable small, and apply Lemma~\ref{lem:extend_Lip_Hol} to
get new functions $f_1^*, f_2^*, g_1^*, g_2^*$. Replacing each $f_i, g_i$
with its starred version in~\eqref{eq:4mixing} introduces an error controlled
by $e^{-\alpha \epsilon n}$, and therefore exponentially small. Thanks
to~\eqref{eq:4mixing} for Lipschitz functions, the remaining difference for
the starred functions is bounded by
\begin{equation*}
  C e^{-cn} \normLip{f_1^*} \normLip{f_2^*} \normLip{g_1^*} \normLip{g_2^*}
  \leq C e^{-cn} e^{4 (1-\alpha)\epsilon n}
  \norm{f_1^*}_{\mathrm{Hol_\alpha}} \norm{f_2^*}_{\mathrm{Hol_\alpha}}
  \norm{g_1^*}_{\mathrm{Hol_\alpha}} \norm{g_2^*}_{\mathrm{Hol_\alpha}},
\end{equation*}
which is exponentially small if $\epsilon$ was chosen small enough at the
beginning of the argument.
\end{proof}

\begin{rmk}
With respect to the proof of Proposition~\ref{prop:random-local-diffeo}, we
would like to remark that probably all of the arguments are well known but
that they had to be adapted to our situation in order to prove H\"older
continuity of $\omega \mapsto \mu_\omega$ and mixing with respect to
Lipschitz functions (instead of functions in $C^1(X)$) along the fibers. In
here, it turned out to be advantageous to first construct the family of
conformal measures and thereafter the invariant functions for the family of
transfer operators as the regularity of $\omega \mapsto h_\omega$ and $\omega
\mapsto \mu_\omega$ are consequences of the regularity of $\omega \mapsto
\nu_\omega$. Moreover, as the arguments essentially depend on the existence
of the cone field, they probably can be easily adapted to other settings. It
is also worth noting that we did not make use of all features of the cone
field as we never touched the convergence of the logarithmic derivatives,
which is provided by the definition of the cones.

Furthermore, we also would like to draw attention
to~\cite{Atnip-Froyland-Vaienti--Thermodynamic-Formalism-For-Random--CMP2021,Atnip-Froyland-Vaienti--Equilibrium-States-For-Non-transitive--2022},
where the authors studied similar cones adapted to random interval
transformations. In there, the cones are defined through the BV-norm instead
of the $C^1(X)$ norm. In particular, provided that $\log \mathcal{L}^n$ is
sufficiently regular (see  Lemma~\ref{lemma:continuity-of-Ln}), the above
proof is applicable in verbatim.

We also would like to point out that the H\"older continuity of $\nu_\omega$
with respect to $\omega$ was obtained
in~\cite{Denker-Gordin--Gibbs-Measures-For-Fibred--AM1999} in a uniformly
expanding setting for the more general setting of fibered systems.
\end{rmk}

\section{Proofs}

\label{sec:proofs}

\subsection{Preliminaries}

Given a space with bounded local complexity, consider for each $r$ a finite
sequence $x^{(r)}_1,\dotsc, x^{(r)}_{k(r)}$ of points as in the definition,
such that the space is covered by the balls $B(x_p^{(r)}, r)$ and such that
no point is in more than $C_0$ balls $B(x_p^{(r)}, 4r)$. Fix also functions
$\rho^{(r)}_p$ supported around $x_p^{(r)}$, which are equal to $1$ on
$B(x_p^{(r)}, 2r)$, to $0$ outside of $B(x_p^{(r)}, 4r)$, and take values in
$[0,1]$. For instance, one can take $\rho^{(r)}_p(x) = \rho( d(x,x_p)/r))$
where $\rho : \R \to \R$ is equal to $1$ on $(-\infty, 2]$, to $0$ on
$[4,\infty)$ and affine in between. With this specific choice, one has a
Lipschitz control
\begin{equation}
\label{eq:normLip_rhop}
  \normLip{\rho_p^{(r)}}\leq \frac{1}{r}
\end{equation}
that will prove useful later.

The main point of these definitions is that one can approximate $(x,y)\mapsto
1_{d(x,y)\le r}$ by a sum of functions $\rho^{(r)}_p$, to which we will be
able to apply mixing arguments:
\begin{lem}
\label{lem:comparison}
For any $x,y\in X$,
\begin{equation}
\label{eq:compare}
  1_{d(x,y)\le r} \leq \sum_{p=1}^{k(r)} \rho^{(r)}_p(x) \rho^{(r)}_p(y) \leq C_0 1_{d(x,y)\le 8 r}.
\end{equation}
\end{lem}
\begin{proof}
Assume $d(x,y)\le r$. There is some $p=p(x, r)$ such that $x\in B(x^{(r)}_p,
r)$, as these balls cover the space. Then $y \in B(x,r) \subseteq
B(x^{(r)}_p, 2r)$. Therefore, $\rho^{(r)}_p(x) = \rho^{(r)}_p(y) = 1$,
proving the left inequality.

Conversely, in the sum $\sum_p \rho^{(r)}_p(x) \rho^{(r)}_p(y)$, a term can
only be nonzero if $d(x,x_p) \le 4r$. There are at most $C_0$ such values of
$p$, by definition of bounded local complexity. For each such $p$, the factor
$\rho^{(r)}_p(y)$ can only be nonzero if $d(y, x_p^{(r)}) \leq 4r$, which
implies $d(x,y)\leq 8r$. This proves the right inequality.
\end{proof}

This lemma makes it possible to express the different correlation dimensions in terms
of the discretization $\rho_p^{(r)}$:
\begin{lem}
\label{lem:rhop_tendsto} One has
\begin{equation*}
  \limsup_{r\to 0} \frac{\displaystyle\log\pare*{\sum_{p=1}^{k(r)} \pare*{\int \rho^{(r)}_p \dd\mu}^2}}{\log r}
  = \oHan,
\end{equation*}
and
\begin{equation*}
  \limsup_{r\to 0} \frac{\displaystyle\log\pare*{\sum_{p=1}^{k(r)} \int \pare*{\int \rho^{(r)}_p \dd\mu_\omega}^2 \dd\Pbb(\omega)}}{\log r}
  = \oHqu,
\end{equation*}
Using liminfs instead, similar equations hold for $\uHan$ and $\uHqu$.
\end{lem}
\begin{proof}
We claim that, for any probability measure $\eta$ on $X$,
\begin{equation*}
  \int \eta(B(x,r)) \dd\eta(x)
  \leq \sum_{p=1}^{k(r)} \pare*{\int \rho^{(r)}_p \dd\eta}^2
  \leq C_0 \int \eta(B(x,8r)) \dd\eta(x).
\end{equation*}
This follows from integrating the inequalities in Lemma~\ref{lem:comparison}
with respect to $\eta^{\otimes 2}$.

Applying these inequalities to the measures $\mu$ or $\mu_\omega$, the lemma
follows readily from the definitions of the correlation dimensions.
\end{proof}

\begin{lem}
\label{lem:k_le} Let $k = k(r)$ be as in
Definition~\ref{def:local_complexity}. Then, for small enough $r$, one has
$k(r) \leq r^{-C'_0}$ for $C'_0 = 4 \log C_0$.
\end{lem}
\begin{proof}
Let us show that
\begin{equation}
\label{eq:bound_k_rec}
  k(r) \leq C_0 k(2r).
\end{equation}
Since the balls $(B(x^{(2r)}_q))_{q\leq k(2r)}$ cover the space, any point
$x_p^{(r)}$ belongs to one of these balls, for some $q=q(p)$. This defines a
map $\{1,\dotsc, k(r)\}\to \{1,\dotsc, k(2r)\}$. Moreover, each $q$ has at
most $C_0$ preimages, by definition of the bounded local complexity. This
shows~\eqref{eq:bound_k_rec}.

We deduce that $k(r) \leq C_0^n k(2^n r)$. Take $n$ so that $2^n r$ is of the
order of magnitude of $1$ (so that $k(2^n r)$ is bounded), e.g., $n =
-\lfloor \log r / \log 2\rfloor$. This gives
\begin{equation*}
  k(r) \leq C C_0^{-\log r/\log 2} = C r^{-\log C_0/\log 2}.
\end{equation*}
As $4\log C_0 > \log C_0/\log 2$, the conclusion follows.
\end{proof}

All the forthcoming proofs will be based on the same scheme. We will fix a
small enough~$r$, and use the functions $\rho_p^{(r)}$, omitting the
superscript ${}^{(r)}$ for readability. We will need to compute the first or
second moment of some functions, to apply Markov inequalities to estimate
ultimately the probability that $m_n$ is smaller or larger than $r$, and then
conclude with a Borel-Cantelli argument. Depending on the precise quantity to
be estimated, we will need stronger or weaker mixing conditions. The
following notations will be used throughout:
\begin{itemize}
\item We define a function $R_p : \Omega \to \R$ by
\begin{equation}
\label{eq:defRp}
  R_p(\omega) = \int \rho_p \dd\mu_\omega.
\end{equation}
As $\rho_p$ is Lipschitz with $\normLip{\rho_p}\leq 1/r$
by~\eqref{eq:normLip_rhop}, the assumption that $\omega \mapsto \mu_\omega$
is Lipschitz ensures that $R_p$ is also Lipschitz, with
\begin{equation}
\label{eq:normLip_Rp}
  \normLip{R_p} \leq C_2/r.
\end{equation}
With this notation, the second part of Lemma~\ref{lem:rhop_tendsto} can be
reformulated as:
\begin{equation}
\label{eq:limsup_Rp}
  \limsup_{r\to 0} \frac{\displaystyle \log \pare*{\sum_p \int R_p^2 \dd\Pbb}}{\log r} = \oHqu,
\end{equation}
and similarly for the liminf.

\item Both the base map $\theta$ and the fiber map $T_\omega$ are mixing
    stretched exponentially, with respective constants $(C_2, c_2)$ and
    $(C_3, c_3)$. We define a function
\begin{equation*}
  \phi(n) = C_\phi \exp(-n^{c_\phi}),
\end{equation*}
with $C_\phi=\max(C_2, C_3)$ and $c_\phi = \min(c_2, c_3)$, that bounds
from above the mixing rate of both maps. Its main property is that, with
our choice of the gap function $\alpha$ in~\eqref{eq:def_alpha}, one has
\begin{equation}
\label{eq:phi_alpha_le}
  \phi(\alpha(n)) \leq \frac{C}{n^{\log n}}.
\end{equation}
In particular, $\phi(\alpha(n))$ tends to zero faster than any polynomial,
which is the property we will use below.
\end{itemize}

\subsection{Upper bounds}
\label{ssubec:proof_upper}

In this paragraph, we prove Propositions~\ref{prop:Mnle_upper_bound}
and~\ref{prop:Mngt_upper_bound}, giving upper bounds for $-\log m_n^\le$ and
$-\log m_n^>$, i.e., showing that respectively on-diagonal and off-diagonal
distances along orbits cannot be too small. The probability that they are
too small will be estimated thanks to a first moment computation, and we will
conclude with a Borel-Cantelli argument. These arguments are pretty soft, as
testified by the fact that the assumptions in these propositions are much
milder than for the corresponding lower bounds.

\begin{proof}[Proof of Proposition~\ref{prop:Mnle_upper_bound} on on-diagonal upper bounds]
First of all, we observe that if $\uHqu = 0$, the theorem is empty. Assume
now that $\uHqu > 0$. Let $\epsilon < \uHqu$. Fix a small enough $r$.

Define a function
\begin{equation*}
  S_n^\le(\omega; x, y) = \sum_{\substack{i, j < n\\ \abs{j-i} \leq \alpha(n)}} \sum_p \rho_p(T_\omega^i x) \rho_p(T_\omega^j y),
\end{equation*}
where $\rho_p = \rho_p^{(r)}$ is the discretization at scale $r$.

If $m_n^\le(\omega; x, y) \le r$, then there are two indices $i, j$ such that
$d(T_\omega^i x, T_\omega^j y) \le r$. Then $S_n^\le(\omega; x, y) \ge 1$ by
Lemma~\ref{lem:comparison}. With Markov's inequality, we obtain
\begin{align*}
  \Pbb\otimes \mu_\omega\otimes \mu_\omega\{(\omega, x, y) \st & m_n^\le(\omega; x, y) \le r\}
  \\& \le \Pbb\otimes \mu_\omega\otimes \mu_\omega\{(\omega, x, y) \st S_n^\le(\omega; x, y) \ge 1\}
  \le \E(S_n^\le).
\end{align*}
Let us now compute this expectation. We have
\begin{align*}
  \E(S_n^\le) &= \sum_p \sum_{\abs{j-i}\le \alpha(n)} \int \rho_p(T_\omega^i x) \rho_p(T_\omega^j y) \dd\mu_\omega(x) \dd\mu_\omega(y) \dd\Pbb(\omega)
  \\& = \sum_p \sum_{\abs{j-i}\le \alpha(n)} \int \pare*{\int \rho_p \dd\mu_{\theta^i \omega}} \pare*{\int \rho_p \dd\mu_{\theta^j \omega}} \dd\Pbb(\omega),
\end{align*}
since $T_\omega^i x$ is distributed according to $\mu_{\theta^i \omega}$ when
$x$ is distributed according to $\mu_\omega$, by equivariance. We recall the
notation $R_p$ introduced in~\eqref{eq:defRp}. Applying the Cauchy-Schwarz
inequality and then using the invariance of $\Pbb$ under $\theta$, we obtain
\begin{align*}
  \E(S_n^\le) & \le \sum_p \sum_{\abs{j-i}\le \alpha(n)} \left[\int R_p(\theta^i \omega)^2 \dd\Pbb(\omega)\right]^{1/2}
            \left[\int R_p(\theta^j \omega)^2 \dd\Pbb(\omega)\right]^{1/2}
  \\& = \sum_p \sum_{\abs{j-i}\le \alpha(n)} \int R_p(\omega)^2 \dd\Pbb(\omega).
\end{align*}

The sum over $i,j$ reduces to $2n \cdot \alpha(n)$. Thanks
to~\eqref{eq:limsup_Rp}, if $r$ is small enough one has $\sum_p \int
R_p(\omega)^2 \dd\Pbb(\omega) \leq r^{\uHqu - \epsilon}$. We have obtained
\begin{equation*}
  \Pbb\otimes \mu_\omega\otimes \mu_\omega\{(\omega, x, y) \st m_n^\le(\omega; x, y) \le r\}
  \le 2n \alpha(n) r^{\uHqu - \epsilon}.
\end{equation*}

Take now $r =r_n = 1/n ^{(1+2\epsilon)/(\uHqu - \epsilon)}$. As $\alpha(n) =
O(n^\epsilon)$, the previous bound gives
\begin{equation*}
  \Pbb\otimes \mu_\omega\otimes \mu_\omega\{(\omega, x, y) \st m_n^\le(\omega; x, y) \le r_n\} = O(1/n^\epsilon).
\end{equation*}
Choose a subsequence $n_s = \lfloor s^{2/\epsilon} \rfloor$. Along this
subsequence, the error probability is $O(1/s^2)$, which is summable. By
Borel-Cantelli, for almost every $(\omega, x, y)$, one has eventually
$m_{n_s}^\le(\omega; x, y) > r_{n_s}$, and therefore
\begin{equation*}
  \frac{-\log m_{n_s}^\le(\omega; x, y)}{\log n_s} \le \frac{1+2\epsilon}{\uHqu - \epsilon}.
\end{equation*}
As $-\log m_n^\le(\omega; x, y)$ is a non-decreasing function of $n$ and
$\log n_{s+1} / \log n_s\to 1$, this inequality along the subsequence $n_s$
passes to the whole sequence. We obtain almost surely
\begin{equation*}
  \limsup \frac{-\log m_n^\le(\omega; x, y)}{\log n} \le \frac{1+2\epsilon}{\uHqu - \epsilon}.
\end{equation*}
As $\epsilon$ is arbitrary, this proves the result.
\end{proof}

\begin{proof}[Proof of Proposition~\ref{prop:Mngt_upper_bound} on  off-diagonal upper bounds]
We follow the same strategy as in the previous proof. The result is obvious
if $\uHan = 0$. Assume that $\uHan > 0$. Let $\epsilon < \uHan$. Fix a small
enough $r$.

Define a function
\begin{equation*}
  S_n^>(\omega; x, y) = \sum_{\substack{i, j < n\\ \abs{j-i} > \alpha(n)}} \sum_p \rho_p(T_\omega^i x) \rho_p(T_\omega^j y).
\end{equation*}
As above, we have
\begin{align*}
  \Pbb\otimes \mu_\omega\otimes \mu_\omega\{(\omega, x, y) \st & m_n^>(\omega; x, y) \le r\}
  \\& \le \Pbb\otimes \mu_\omega\otimes \mu_\omega\{(\omega, x, y) \st S_n^>(\omega; x, y) \ge 1\}
  \le \E(S_n^>).
\end{align*}
Moreover,
\begin{equation*}
  \E(S_n^>)
  = \sum_{\substack{i, j < n\\ \abs{j-i} > \alpha(n)}} \sum_p \int R_p(\theta^i \omega) R_p(\theta^j \omega) \dd\Pbb(\omega),
\end{equation*}
as in the proof of Proposition~\ref{prop:Mnle_upper_bound}.

As the speed of mixing of $\theta$ is at least $\phi$ by definition, we have
\begin{equation*}
\int R_p(\theta^i \omega) R_p(\theta^j \omega) \dd\Pbb(\omega) = \pare*{\int
R_p(\omega)\dd\Pbb(\omega)}^2 + O(\phi(\alpha(n)) r^{-2}),
\end{equation*}
as $R_p$ has a Lipschitz norm bounded by $C_2/r$ (see~\eqref{eq:normLip_Rp}),
and $i$ and $j$ are separated by at least $\alpha(n)$.

Summing over $i, j$ (there are less than $n^2$ of them) and then $p$ (there are $k(r)$
of them), we get
\begin{equation*}
  \E(S_n^>) \leq n^2 \sum_p \pare*{\int R_p(\omega)\dd\Pbb(\omega)}^2 + C n^2 k(r) \phi(\alpha(n)) r^{-2}.
\end{equation*}
Note that $\int R_p(\omega)\dd\Pbb(\omega) = \int \rho_p \dd\mu$. By
Lemma~\ref{lem:rhop_tendsto}, we have for small enough $r$
\begin{equation*}
  \sum_p \pare*{\int \rho_p \dd\mu}^2 \leq r^{\uHan - \epsilon}.
\end{equation*}
Moreover, $k(r) \leq r^{-C'_0}$ by Lemma~\ref{lem:k_le}. Finally,
\begin{equation*}
  \Pbb\otimes \mu_\omega\otimes \mu_\omega\{(\omega, x, y) \st m_n^>(\omega; x, y) \le r\}
  \le n^2 r^{\uHan - \epsilon} + C n^2 r^{-(C'_0 + 2)} \phi(\alpha(n)).
\end{equation*}

Take now $r =r_n = 1/n ^{(2+\epsilon)/(\uHan - \epsilon)}$. As
$\phi(\alpha(n)) \leq C/n^{\log n}$ by~\eqref{eq:phi_alpha_le}, this gives
\begin{equation*}
  \Pbb\otimes \mu_\omega\otimes \mu_\omega\{(\omega, x, y) \st m_n^>(\omega; x, y) \le r_n\}
  \le 1/n^\epsilon + C' n^{C(\epsilon)} / n^{\log n}.
\end{equation*}
For large enough $n$, this is bounded by $2/n^\epsilon$.

Choose a subsequence $n_s = \lfloor s^{2/\epsilon} \rfloor$. Along this
subsequence, the error probability is $O(1/s^2)$, which is summable. By
Borel-Cantelli, for almost every $(\omega, x, y)$, one has eventually
$m_{n_s}^>(\omega; x, y) > r_{n_s}$, and therefore
\begin{equation*}
  \frac{-\log m_{n_s}^>(\omega; x, y)}{\log n_s} \le \frac{2+\epsilon}{\uHan - \epsilon}.
\end{equation*}
Contrary to the proof of Proposition~\ref{prop:Mnle_upper_bound}, we can not
argue from there by monotonicity, as $-\log m_n^>$ is not a non-decreasing
function of $n$ (we are considering more times, but the constraint
$\abs{j-i}> \alpha(n)$ becomes stronger when $n$ increases). What is true,
though, is that for all $n\in [n_s, n_{s+1}]$, one has
\begin{equation*}
  -\log m_n^>(\omega; x, y) \le -\log \min_{\substack{i, j < n_{s+1}\\ \abs{j-i} > \alpha(n_s)}} d(T_\omega^i x, T_\omega^j y) \eqqcolon -\log m'_{n_{s+1}}.
\end{equation*}
One can show exactly as above that, almost surely, eventually,
\begin{equation*}
  \frac{-\log m'_{n_s}(\omega; x, y)}{\log n_s} \le \frac{2+\epsilon}{\uHan - \epsilon}.
\end{equation*}
With the previous equation, this inequality passes to the whole sequence
$m_n^>$. We obtain almost surely
\begin{equation*}
  \limsup \frac{-\log m_n^>(\omega; x, y)}{\log n} \le \frac{2 + \epsilon}{\uHan - \epsilon}.
\end{equation*}
As $\epsilon$ is arbitrary, this proves the result.
\end{proof}

\begin{rmk}
\label{rmk:polynomial_enough1} Stretched exponential mixing is not essential
for the proof of Proposition~\ref{prop:Mngt_upper_bound}: using a gap size
$\alpha(n) = n^\delta$ for an arbitrarily small $\delta$ instead of
$\alpha(n) = (\log n)^{C_4}$ as we did, the proof goes through if the mixing
speed of $\theta$ is faster than any polynomial.
Therefore,~\eqref{eq:upper_bound} holds under this assumption. The stretched
exponential mixing is however necessary in our argument for the lower bound
for $m_n$ (see the proof of Proposition~\ref{prop:Mngg_lower_bound}).
\end{rmk}

\subsection{Lower bounds}
\label{subsec:proof_lower}

In this paragraph, we prove Propositions~\ref{prop:Mn0_lower_bound}
and~\ref{prop:Mngg_lower_bound}, giving lower bounds for $-\log m_n^0$ and
$-\log m_n^\gg$, i.e., showing that respectively on-diagonal and off-diagonal
distances along orbits can not be too large. The probability that they are
too large will be estimated thanks to a second moment computation: we will
control the average of a suitable function, and its variance to show that
this function can not deviate much from its average. We will then conclude
with a Borel-Cantelli argument. These arguments are more technical than the
ones for the corresponding upper bounds, as we need estimates on moments of
order two instead of one (and therefore stronger mixing assumptions to be
able to deal with the more involved terms that show up).

\begin{proof}[Proof of Proposition~\ref{prop:Mn0_lower_bound} on on-diagonal lower bounds.]
Let $\epsilon > 0$. Fix a small enough $r$.

Define a function
\begin{equation*}
  S_n^0(\omega; x, y) = \sum_{i < n} \sum_p \rho_p(T_\omega^i x) \rho_p(T_\omega^i y),
\end{equation*}
where $\rho_p = \rho_p^{(r)}$ is the discretization at scale $r$.

If $m^0_n(\omega; x, y) > 8r$, then for all $i$ one has $d(T_\omega^i x,
T_\omega^i y) > 8r$. Then $S_n^0(\omega; x, y) = 0$ by
Lemma~\ref{lem:comparison}. We obtain
\begin{align*}
  \Pbb\otimes \mu_\omega\otimes \mu_\omega\{(\omega, x, y) \st & m_n^0(\omega; x, y) > 8r\}
  \\& \le \Pbb\otimes \mu_\omega\otimes \mu_\omega\{(\omega, x, y) \st S_n^0(\omega; x, y) = 0\}.
\end{align*}
When $S_n^0 = 0$, then $(S_n^0 - \E(S_n^0))^2 / \E(S_n^0)^2 = 1$. With Markov
inequality, this gives
\begin{equation}
\label{eq:Pbb_le_lmjhsmld}
  \Pbb\otimes \mu_\omega\otimes \mu_\omega\{(\omega, x, y) \st m_n^0(\omega; x, y) > 8r\} \leq
  \frac{\var(S_n^0)}{\E(S_n^0)^2} = \frac{\E((S_n^0)^2) - \E(S_n^0)^2}{\E(S_n^0)^2}.
\end{equation}

We have
\begin{align*}
  \E(S_n^0) &= \sum_p \sum_{i} \int \rho_p(T_\omega^i x) \rho_p(T_\omega^i y) \dd\mu_\omega(x) \dd\mu_\omega(y) \dd\Pbb(\omega)
  \\& = \sum_p \sum_i \int R_p(\theta^i \omega) \cdot R_p(\theta^i \omega)\dd\Pbb(\omega)
  \\& = n \sum_p \int R_p(\omega)^2 \dd\Pbb(\omega),
\end{align*}
by $\theta$-invariance of $\Pbb$. Therefore, its asymptotic behavior is
described by~\eqref{eq:limsup_Rp}.

Let us now estimate $\E((S_n^0)^2)$. Expanding the square, we get
\begin{equation*}
  \E((S_n^0)^2) = \sum_{i,i'} \sum_{p,q} \int  \rho_p(T_\omega^i x) \rho_p(T_\omega^i y)  \rho_q(T_\omega^{i'} x) \rho_q(T_\omega^{i'} y)
    \dd\mu_\omega(x) \dd\mu_\omega(y) \dd\Pbb(\omega).
\end{equation*}
We split the sum according to whether $\abs{i'-i}\le \alpha(n)$ or
$\abs{i'-i}> \alpha(n)$. In the former case, we will use a crude upper bound,
and in the latter we will use mixing. Let us fix $i,i'$.

We have $\sum_q \rho_q(T_\omega^{i'} x) \rho_q(T_\omega^{i'} y) \leq C_0$,
by~\eqref{eq:compare}. Therefore, we get a bound
\begin{equation*}
  C_0 \sum_p \int \rho_p(T_\omega^i x) \rho_p(T_\omega^i y) \dd\mu_\omega(x) \dd\mu_\omega(y) \dd\Pbb(\omega)
  = C_0\sum_p \int R_p(\omega)^2 \dd\Pbb(\omega).
\end{equation*}
We will only use this crude bound when $\abs{i'-i}\le \alpha(n)$. Therefore,
we will get $2n \alpha(n)$ of them.

Assume now $\abs{i'-i} > \alpha(n)$. Then
\begin{equation*}
  \int \rho_p(T_\omega^i x) \rho_q(T_\omega^{i'} x) \dd\mu_\omega(x)
  = \pare*{\int\rho_p \dd\mu_{\theta^i \omega}} \pare*{\int\rho_q \dd\mu_{\theta^{i'} \omega}}
  + O (\phi(\alpha(n)) \normLip{\rho_p} \normLip{\rho_q}),
\end{equation*}
by fiberwise mixing (Definition~\ref{def:mixing_fiber}). As $\normLip{\rho_p}
\leq 1/r$, the contribution of the error terms to $\E((S_n^0)^2)$ is bounded
by
\begin{equation}
\label{eq:sum_error_terms}
  C \sum_{i, i'} \sum_{p, q} \phi(\alpha(n)) r^{-2}
  \leq C n^2 r^{-2C'_0-2} \phi(\alpha(n)),
\end{equation}
as there are $n$ possible values of $i, i'$, and $k(r) \le r^{-C'_0}$
possible values of $p$, $q$.

Let us use the same estimate for the integral with respect to
$\dd\mu_\omega(y)$. The remaining term is
\begin{equation*}
  \sum_{p,q} \sum_{\abs{i'-i} > \alpha(n)} \int R_p(\theta^i \omega)^2 R_q(\theta^{i'} \omega)^2 \dd\Pbb(\omega).
\end{equation*}
The function $R_p$ is Lipschitz, with $\normLip{R_p} \leq C_1 / r$,
see~\eqref{eq:normLip_Rp}. As $R_p$ is bounded by $1$, it follows that
$R_p^2$ is also Lipschitz, with $\normLip{R_p^2} \leq 2C_1/r$. One may
therefore use the mixing of the base transformation $\theta$, to obtain
\begin{equation*}
  \int R_p(\theta^i \omega)^2 R_q(\theta^{i'} \omega)^2 \dd\Pbb(\omega)
  = \pare*{\int R_p^2 \dd\Pbb} \pare*{\int R_q^2 \dd\Pbb} + O(\phi(\alpha(n))r^{-2}).
\end{equation*}
The error terms add up exactly as in~\eqref{eq:sum_error_terms}.

Adding up all the terms, we are left with
\begin{align*}
  \E((S_n^0)^2) & \leq n^2 \sum_{p,q} \pare*{\int R_p^2 \dd\Pbb} \pare*{\int R_q^2 \dd\Pbb}
    + C n^2 r^{-2C'_0-2} \phi(\alpha(n))
    + C n \alpha(n) \left[\sum_p \int R_p^2 \dd\Pbb\right]
  \\& = \E(S_n^0)^2 + C n^2 r^{-2C'_0-2} \phi(\alpha(n)) + C \alpha(n) \E(S_n^0).
\end{align*}
Together with~\eqref{eq:Pbb_le_lmjhsmld}, this yields
\begin{equation}
\label{eq:<lkjmlkjwxmlcvj}
  \Pbb\otimes \mu_\omega\otimes \mu_\omega\{(\omega, x, y) \st m_n^0(\omega; x, y) > 8r\}
  \le
  \frac{C n^2 r^{-2C'_0-2} \phi(\alpha(n))}{\E(S_n^0)^2}
  + \frac{C \alpha(n)}{\E(S_n^0)}.
\end{equation}

For small enough $r$, the limit~\eqref{eq:normLip_Rp} ensures that
\begin{equation*}
  \sum_p \int R_p^2 \dd\Pbb \geq r^{\oHqu + \epsilon}.
\end{equation*}
Take $r = r_n = 1/n^{(1-2\epsilon)/(\oHqu + \epsilon)}$. For this value of
$r$, we get
\begin{equation*}
  \E(S_n^0) = n \sum_p \int R_p^2 \dd\Pbb \geq n \cdot n^{-(1-2\epsilon)} = n^{2\epsilon}.
\end{equation*}
As $\phi(\alpha(n)) = O(1/n^{\log n})$ by~\eqref{eq:phi_alpha_le}, the first
term in~\eqref{eq:<lkjmlkjwxmlcvj} decays faster than any polynomial.
Moreover, as $\alpha(n) = (\log n)^{C_4}$, the second term decays at least
like $1/n^\epsilon$. For large enough $n$, we obtain
\begin{equation*}
  \Pbb\otimes \mu_\omega\otimes \mu_\omega\{(\omega, x, y) \st m_n^0(\omega; x, y) > 8r_n\}
  \le \frac{1}{n^\epsilon}.
\end{equation*}

Choose a subsequence $n_s = \lfloor s^{2/\epsilon} \rfloor$. Along this
subsequence, the error probability is $O(1/s^2)$, which is summable. By
Borel-Cantelli, for almost every $(\omega, x, y)$, one has eventually
$m_{n_s}^0(\omega; x, y) \le 8 r_{n_s}$, and therefore
\begin{equation*}
  \liminf \frac{-\log m_{n_s}^0(\omega; x, y)}{\log n_s} \ge \frac{1-2\epsilon}{\oHqu + \epsilon}.
\end{equation*}
By monotonicity of $m_n^0$, this behavior passes to the whole sequence. As
$\epsilon$ is arbitrary, this concludes the proof of the proposition.
\end{proof}
\begin{rmk}
As in Remark~\ref{rmk:polynomial_enough1}, a mixing rate faster than any
polynomial would be enough for the proof of
Proposition~\ref{prop:Mn0_lower_bound}.
\end{rmk}

\begin{proof}[Proof of Proposition~\ref{prop:Mngg_lower_bound} on off-diagonal lower bounds.]
Fix a small enough $r$. In this proof, we will consider indices $i$ or $i'$
that will always be restricted to the range $[0, n/3)$, and indices $j$ or
$j'$ that will always be restricted to the range $[2n/3, n)$. Instead of
always specifying this, we will use the notation $\sum'$ to enforce these
restrictions implicitly. In this proof, an error term (depending on $n$ and
$r$) will be called \emph{admissible} if it is bounded by $r^{-C} u(n)$, for
some $C$ and some function $u$ that tends to zero faster than any $n^{-D}$. A
generic admissible error term will be denoted by $a_n$, keeping $r$ implicit.

Define a function
\begin{equation*}
  S_n^\gg(\omega; x, y) = \sumprime_{i, j} \sum_p \rho_p(T_\omega^i x) \rho_p(T_\omega^j y).
\end{equation*}
As in~\eqref{eq:Pbb_le_lmjhsmld}, we have
\begin{equation}
\label{eq:Pbb_gg}
  \Pbb\otimes \mu_\omega\otimes \mu_\omega\{(\omega, x, y) \st m_n^\gg(\omega; x, y) > 8r\} \leq
  \frac{\var(S_n^\gg)}{\E(S_n^\gg)^2} = \frac{\E((S_n^\gg)^2) - \E(S_n^\gg)^2}{\E(S_n^\gg)^2}.
\end{equation}

Let us first estimate $\E(S_n^\gg)$. We have
\begin{align*}
  \E(S_n^\gg) &= \sumprime_{i, j} \sum_p \int \pare*{\int \rho_p(T_\omega^i x) \dd\mu_\omega(x)}
    \pare*{\int \rho_p(T_\omega^j y) \dd\mu_\omega(y)}\dd\Pbb(\omega)
  \\& = \sumprime_{i,j} \sum_p \int R_p(\theta^i \omega) R_p(\theta^j \omega) \dd\Pbb(\omega).
\end{align*}
By~\eqref{eq:normLip_Rp}, the function $R_p$ is Lipschitz, with
$\normLip{R_p} \leq C_2/r$. The mixing of $\theta$ gives then
\begin{equation*}
  \int R_p(\theta^i \omega) R_p(\theta^j \omega) \dd\Pbb(\omega)
  = \pare*{\int R_p \dd\Pbb}^2 + O(\phi(n/3) r^{-2}),
\end{equation*}
as the gap between $i$ and $j$ is at least $n/3$. Summing over $n$, and
writing
\begin{equation}
\label{eq:defA}
  A_n = \sum_p \pare*{\int \rho_p \dd\mu}^2 = \sum_p \pare*{\int R_p \dd\Pbb}^2,
  \quad B_n = (n/3)^2 A_n,
\end{equation}
we get
\begin{equation*}
  \E(S_n^\gg) = (n/3)^2 A_n + O(n^2 \phi(n/3) r^{-2})
  = B_n + O(a_n),
\end{equation*}
where we recall that $a_n$ is a generic admissible error term. Note that
$B_n$ also depends on $r$, but we keep this implicit in the notation. Also,
$A_n$ only depends on $r$ and not on $n$, but since in the end we want to let
$r$ depend on $n$ we keep the index $n$.

\medskip

Let us now estimate $\E((S_n^\gg)^2)$, up to an admissible error term. We
expand the square, getting $6$ indices $i, i', j, j', p, q$.
\begin{equation*}
  \E((S_n^\gg)^2)
  = \sumprime_{i,i',j,j'} \sum_{p,q} \int \rho_p(T_\omega^i x) \rho_q(T_\omega^{i'} x) \rho_p(T_\omega^j y) \rho_q(T_\omega^{j'}y) \dd\mu_\omega(x) \dd\mu_\omega(y)\dd\Pbb(\omega).
\end{equation*}
We will split this sum depending on whether $i$ and $i'$ are close, i.e.,
$\abs{i'-i}\le \alpha(n)$, or whether they are far, and similarly for $j$ and
$j'$.

Assume first that $i$ and $i'$ are far. Then
\begin{equation}
\label{eq:lkjwxmlkjwcv}
  \int \rho_p(T_\omega^i x) \rho_q(T_\omega^{i'} x) \dd\mu_\omega(x)
  = R_p(\theta^i \omega) R_q(\theta^{i'} \omega) + O(\phi(\alpha(n))r^{-2}),
\end{equation}
thanks to the fiberwise mixing (Definition~\ref{def:mixing_fiber}). When
adding these error terms over all possible $i,i',j, j', p, q$, one gets an
error at most
\begin{equation*}
  C \phi(\alpha(n))r^{-2} \cdot n^4 k(r)^2
  \leq C n^4 r^{-2C'_0-2} \phi(\alpha(n)),
\end{equation*}
which is admissible as $\phi(\alpha(n))\leq C/n^{\log n}$.

In the same way, when $j$ and $j'$ are far, we can replace $\int
\rho_p(T_\omega^j y) \rho_q(T_\omega^{j'} y) \dd\mu_\omega(y)$ with
$R_p(\theta^j \omega)R_q(\theta^{j'}\omega)$, up to an admissible error.

\medskip

\emph{Case 1: when $(i,i')$ are far away and $(j,j')$ are far away.}

Up to an admissible error, the contribution of these terms to
$\E((S_n^\gg)^2)$ is bounded by
\begin{equation*}
   \sumprime_{\abs{i'-i} > \alpha(n), \abs{j'-j}>\alpha(n)} \sum_{p,q}
    \int R_p(\theta^i \omega) R_q(\theta^{i'} \omega) R_p(\theta^j \omega) R_q(\theta^{j'} \omega) \dd\Pbb(\omega).
\end{equation*}
There is a gap of $n/3$ between $\max\{i,i'\}$ and $\min\{j,j'\}$. Therefore, using
$4$-mixing, we can replace the above integral with
\begin{equation*}
  \pare*{\int R_p(\theta^i \omega) R_q(\theta^{i'} \omega) \dd\Pbb(\omega)}
  \pare*{\int R_p(\theta^j \omega) R_q(\theta^{j'} \omega) \dd\Pbb(\omega)}
\end{equation*}
up to an error of $C \phi(n/3) r^{-4}$. The sum of these errors over
$p,q,i,i',j,j'$ is admissible. Then, using $2$-mixing, one can replace $\int
R_p(\theta^i \omega) R_q(\theta^{i'} \omega) \dd\Pbb(\omega)$ with $\int R_p
\cdot \int R_q$, up to an error which is again admissible as the gap between
$i$ and $i'$ is at least $\alpha(n)$. The same goes for $\int R_p(\theta^j
\omega) R_q(\theta^{j'} \omega) \dd\Pbb(\omega)$.

Finally, up to an admissible error, the contribution of this case to
$\E((S_n^\gg)^2)$ is
\begin{equation*}
  \sumprime_{\abs{i'-i} > \alpha(n), \abs{j'-j}>\alpha(n)} \sum_{p,q}  \pare*{\int R_p \dd\Pbb}^2 \pare*{\int R_q \dd\Pbb}^2
  \leq \Bigl((n/3)^2 A_n\Bigr)^2 = B_n^2,
\end{equation*}
as $\int R_p \dd\Pbb = \int \rho_p \dd\mu$ by definition of $\mu$. Recall
that, with $A_n$ and $B_n$ defined as in~\eqref{eq:defA}, then $B_n$ is the
dominant term in the expansion of $\E(S_n^\gg)$.

\medskip

\emph{Case 2: when $(i,i')$ are far away and $(j,j')$ are close (or
conversely).}

By symmetry, we assume that $(i,i')$ are far away and $(j,j')$ are close. In
this case, after doing the substitution~\eqref{eq:lkjwxmlkjwcv}, we should
study
\begin{equation}
\label{eq:lkwjxmcvmlkjlmwxjcv}
  \int R_p(\theta^i \omega) R_q(\theta^{i'} \omega) \pare*{\int \rho_p(T_\omega^j y) \rho_q(T_\omega^{j'} y)\dd\mu_\omega(y)} \dd\Pbb(\omega).
\end{equation}
Assume for instance $j \le j'$. Then, changing variables with $y'= T_\omega^j
y$, and writing $\omega' = \theta^j \omega$, the inner integral becomes
\begin{equation*}
  F(\omega') = \int \rho_p (y') \rho_q(T_{\omega'}^{j'-j} y') \dd\mu_{\omega'}(y').
\end{equation*}
We claim that this function $F$ is Lipschitz continuous, with
\begin{equation}
\label{eq:normLip_F}
  \normLip{F} \leq C C_S^{\alpha(n)} r^{-2},
\end{equation}
where $C_S\ge 1$ is a Lipschitz constant for $S$. To prove this, let us
compute
\begin{multline}
\label{eq:F_sub_F}
  F(\omega_1) - F(\omega_2) = \left[\int \rho_p \cdot (\rho_q \circ T_{\omega_1}^{j'-j} - \rho_q \circ T_{\omega_2}^{j'-j}) \dd\mu_{\omega_1}\right]
  \\ + \left[ \int \rho_p \cdot \rho_q \circ T_{\omega_2}^{j'-j} \dd\mu_{\omega_1} - \int \rho_p \cdot \rho_q \circ T_{\omega_2}^{j'-j} \dd\mu_{\omega_2}\right].
\end{multline}
For any $y$, we have
\begin{equation*}
  \rho_q \circ T_{\omega_1}^{j'-j}(y) - \rho_q \circ T_{\omega_2}^{j'-j}(y)
  = \rho_q\circ \pi_2 (S^{j'-j} (\omega_1, y)) - \rho_q \circ \pi_2 (S^{j'-j} (\omega_2, y)),
\end{equation*}
where $\pi_2 : \Omega \times X \to X$ is the second projection. As $S$ is
Lipschitz with Lipschitz constant $C_S$, we have $d(S^{j'-j}(\omega_1, y),
S^{j'-j}(\omega_2, y)) \leq C_S^{j'-j} d(\omega_1, \omega_2)$. Therefore, the
term on the first line of~\eqref{eq:F_sub_F} is bounded by $C_S^{\alpha(n)}
r^{-1} d(\omega_1, \omega_2)$, as $\rho_q$ has Lipschitz constant at most
$r^{-1}$ and $j'-j \le \alpha(n)$. For the term on the second line, we note
that $\rho_p \cdot \rho_q \circ T_{\omega_2}^{j'-j}$ is Lipschitz, with
Lipschitz constant at most $r^{-2} C_S^{\alpha(n)}$, by the same argument.
Since $\omega' \mapsto \mu_{\omega'}$ is Lipschitz
(Definition~\ref{def:lipschitz_fibers}), it follows that this term is bounded
by $C r^{-2} C_S^{\alpha(n)} d(\omega_1, \omega_2)$. This completes the proof
of~\eqref{eq:normLip_F}.

We can write~\eqref{eq:lkwjxmcvmlkjlmwxjcv} as $\int R_p(\theta^i \omega)
R_q(\theta^{i'} \omega) F(\theta^{j} \omega) \dd\Pbb(\omega)$. Using
$3$-mixing (which follows from $4$-mixing by taking the last function equal
to $1$), this is equal to
\begin{equation*}
  \pare*{\int R_p(\theta^i \omega) R_q(\theta^{i'} \omega) \dd\Pbb(\omega)} \pare*{\int F \dd\Pbb}
  + O(\normLip{R_p} \normLip{R_q} \normLip{F} \phi(n/3)).
\end{equation*}
The error terms add up to at most $n^4 r^{-4-2C'_0} C_S^{\alpha(n)}
\phi(n/3)$, thanks to~\eqref{eq:normLip_F}. As $\phi(n/3) \leq C
e^{-(n/3)^c}$ for some $c>0$ while $\alpha(n) = (\log n)^{C_4}$, this error
term is again admissible.

We can also replace $\int R_p(\theta^i \omega) R_q(\theta^{i'} \omega)
\dd\Pbb(\omega)$ with $\pare*{\int R_p} \pare*{\int R_q}$ up to an admissible
error term, thanks to the mixing of $\theta$ and since $i$ and $i'$ are far
apart. Finally, up to an admissible error term, the contribution of these
terms to $\E((S_n^\gg)^2)$ is at most
\begin{align*}
  &\sum_{p,q} \pare*{\int R_p \dd\Pbb} \cdot \pare*{\int R_q \dd\Pbb}
  \cdot \pare*{\int \rho_p (y) \rho_q(T_{\omega}^{j'-j} y) \dd\mu_\omega(y) \dd\Pbb(\omega)}
  \\&= \int \pare*{\sum_p \left[\int R_p \dd\Pbb\right] \rho_p(y)} \pare*{\sum_q \left[\int R_q \dd\Pbb\right] \rho_q(T_{\omega}^{j'-j} y)} \dd\mu_\omega(y) \dd\Pbb(\omega)
  \\& \leq \pare*{ \int \pare*{\sum_p \left[\int R_p \dd\Pbb\right] \rho_p(y)}^2 \dd\mu_\omega(y) \dd\Pbb(\omega)}^{1/2}
  \\& \quad \quad \quad\quad \times \pare*{ \int \pare*{\sum_q \left[\int R_q \dd\Pbb\right] \rho_q(T_{\omega}^{j'-j} y)}^2 \dd\mu_\omega(y) \dd\Pbb(\omega)}^{1/2},
\end{align*}
by Cauchy-Schwartz. By invariance of the measure $\Pbb\otimes \mu_\omega$
under $S$, the two factors in the last product coincide, eliminating the
square roots.

For any $y$, there are at most $C_0$ nonzero terms in the sum $\sum_p
\left[\int R_p \dd\Pbb\right] \rho_p(y)$, by bounded local complexity. We can
therefore use the convexity inequality $(a_1+\dotsc+a_{C_0})^2 \le C_0 (a_1^2
+ \dotsc+a_{C_0}^2)$ to bound the square of the sum by the sum of the
squares. We get a bound
\begin{equation*}
  C_0 \int \sum_p \left[\int R_p \dd\Pbb\right]^2 \rho_p(y)^2 \dd\mu_\omega(y) \dd\Pbb(\omega).
\end{equation*}
Bounding $\rho_p^2$ by $\rho_p$, and since $\int \rho_p(y) \dd\mu_\omega(y)
\dd\Pbb(\omega) = \int R_p \dd\Pbb$, we are left with a bound
\begin{equation*}
  C_0 \sum_p \left[\int R_p \dd\Pbb\right]^3.
\end{equation*}
By concavity of the function $x\mapsto x^{2/3}$, one has the inequality
$\pare*{\sum a_i}^{2/3} \le \sum a_i^{2/3}$, and therefore $\sum a_i \le
\pare*{\sum a_i^{2/3}}^{3/2}$. Applying this inequality to the previous
equation, one obtains a bound
\begin{equation*}
  C_0 \pare*{\sum_p \left[\int R_p \dd\Pbb\right]^2}^{3/2}
  = C_0 A_n^{3/2}.
\end{equation*}
Summing over the possible values of $i, i', j$ (there are $(n/3)^3$ of them)
and then $j'$ (there are at most $2\alpha(n)$ of them, as $\abs{j'-j} \leq
\alpha(n)$), we get that the total contribution of this case is bounded by
\begin{equation*}
  C \alpha(n)n^3 A_n^{3/2} = C \alpha(n) B_n^{3/2},
\end{equation*}
up to an admissible error term.

\medskip

\emph{Case 3: when both $(i,i')$ and $(j,j')$ are close.}

In this case, we can use rough estimates. We have
\begin{multline*}
  \sum_{p,q} \int \rho_p(T_\omega^i x) \rho_q(T_\omega^{i'} x) \rho_p(T_\omega^j y) \rho_q(T_\omega^{j'}y) \dd\mu_\omega(x) \dd\mu_\omega(y)\dd\Pbb(\omega)
  \\ \le C_0\sum_p  \int \rho_p(T_\omega^i x) \rho_p(T_\omega^j y) \dd\mu_\omega(x) \dd\mu_\omega(y)\dd\Pbb(\omega),
\end{multline*}
as $\sum_q \rho_q (z) \rho_q(z')\le C_0$ for any $z, z'$,
by~\eqref{eq:compare}. This is equal to $C_0\int R_p (\theta^i \omega)
R_p(\theta^j \omega) \dd\Pbb(\omega)$, which coincides with $C_0\pare*{\int
R_p \dd\Pbb}^2$ by mixing, up to an error term which is admissible, as above.
Finally, the total contribution of this case is, up to an admissible error
term, bounded by
\begin{equation*}
  C n^2 \alpha(n)^2 \sum_p \pare*{\int R_p \dd\Pbb}^2
  = C \alpha(n)^2 B_n.
\end{equation*}

\medskip

\emph{Conclusion.}

Combining all three cases, we get
\begin{equation*}
  \E((S_n^\gg)^2) \leq B_n^2 + C \alpha(n) B_n^{3/2} + C \alpha(n)^2 B_n + a_n,
\end{equation*}
where $a_n$ is an admissible error term.

Let $\epsilon > 0$. For small enough $r$, the
convergence~\eqref{eq:limsup_Rp} ensures that, for large $n$,
\begin{equation*}
  A_n = \sum_p \pare*{\int R_p \dd\Pbb}^2 \geq r^{\oHan + \epsilon}.
\end{equation*}
Take $r = r_n = 1/n^{(2-3\epsilon)/(\oHan + \epsilon)}$. Then $n^2 A_n \geq
n^{3\epsilon}$. Note that $\E(S_n^\gg) = B_n + a_n$ where $a_n$ is an
admissible error term. We get $\E(S_n^\gg)^2 = B_n^2 + a_n^2 + 2B_n a_n$,
where both $a_n^2$ and $2B_n a_n$ are again admissible error terms, and may
thus be written as $a'_n$. In particular, $\E(S_n^\gg)^2$ is asymptotic to
$B_n^2$, and is larger than $B_n^2/2$ for large $n$. With~\eqref{eq:Pbb_gg},
we get
\begin{align*}
  \Pbb\otimes \mu_\omega\otimes &\mu_\omega\{(\omega, x, y) \st m_n^\gg(\omega; x, y) > 8r\} \leq
  \frac{\E((S_n^\gg)^2) - \E(S_n^\gg)^2}{\E(S_n^\gg)^2}
\\&
  \leq \frac{B_n^2 + C \alpha(n) B_n^{3/2} + C \alpha(n)^2 B_n + a_n - (B_n + a_n)^2}{B_n^2/2}
\\&
  = \frac{B_n^2 + C \alpha(n) B_n^{3/2} + C \alpha(n)^2 B_n + a_n - B_n^2 + a'_n}{B_n^2/2}
\\&
  \leq C \frac{\alpha(n)}{B_n^{1/2}} + C \frac{\alpha(n)^2}{B_n} + C \frac{a_n + a'_n}{B_n^2}.
\end{align*}
As $\alpha(n)$ grows more slowly than any polynomial while $B_n \geq
n^{3\epsilon}$, it follows that the above probability is $O(n^{-\epsilon})$.

\medskip

Choose a subsequence $n_s = \lfloor s^{2/\epsilon} \rfloor$. Along this
subsequence, the error probability is $O(1/s^2)$, which is summable. By
Borel-Cantelli, for almost every $(\omega, x, y)$, one has eventually
$m_{n_s}^\gg(\omega; x, y) \le 8 r_{n_s}$, and therefore
\begin{equation*}
  \liminf \frac{-\log m_{n_s}^\gg(\omega; x, y)}{\log n_s} \ge \frac{2-3\epsilon}{\oHan + \epsilon}.
\end{equation*}
One can not conclude directly from this, as the sequence $m_n^\gg$ is not
monotone. What is true, though, is that for all $n\in [n_s, n_{s+1}]$, one
has
\begin{equation*}
  -\log m_n^\gg(\omega; x, y) \ge -\log \min_{i < n_s/3,\ 2n_{s+1}/3 \le j < n_s} d(T_\omega^i x, T_\omega^j y) \eqqcolon -\log m'_{n_s}.
\end{equation*}
One can show exactly as above that, almost surely,
\begin{equation*}
  \liminf \frac{-\log m'_{n_s}(\omega; x, y)}{\log n_s} \ge \frac{2-3\epsilon}{\oHan + \epsilon}.
\end{equation*}
With the previous equation, this inequality passes to the whole sequence
$m_n^\gg$. We obtain almost surely
\begin{equation*}
  \liminf \frac{-\log m_n^\gg(\omega; x, y)}{\log n} \ge \frac{2 - 3\epsilon}{\oHan + \epsilon}.
\end{equation*}
As $\epsilon$ is arbitrary, this proves the result.
\end{proof}

\end{document}